\documentclass [twoside,reqno,12pt] {amsart}
\usepackage[left=1in,right=1in,top=1in,bottom=1in]{geometry}

\usepackage[hidelinks]{hyperref}
\usepackage{amsfonts}
\usepackage{amssymb}
\usepackage{color}
\usepackage{graphics}
\usepackage{comment}

\newtheorem{thm}{Theorem}[section]
\newtheorem{cor}[thm]{Corollary}
\newtheorem{lem}[thm]{Lemma}
\newtheorem{prop}[thm]{Proposition}

\newtheorem{rem}[thm]{Remark}

\theoremstyle{definition}

\numberwithin{equation}{section}

\renewcommand{\Re}{\mathrm{Re}}
\renewcommand{\Im}{\mathrm{Im}}

\newcommand{\C}{\mathbb{C}}

\renewcommand{\div}{\operatorname{div}}

\newcommand{\R}{\mathbb{R}}

\newcommand{\scl}{\mathrm{scl}}

\def\tilde{\widetilde}
\def \bfo {\begin {eqnarray*} }
\def \efo {\end {eqnarray*} }
\def \ba {\begin {eqnarray*} }
\def \ea {\end {eqnarray*} }
\def \beq {\begin {eqnarray}}
\def \eeq {\end {eqnarray}}
\def \supp {\hbox{supp }}

\def \p {\partial}

\usepackage{empheq}

\def\tilde{\widetilde}
\def \bfo {\begin {eqnarray*} }
\def \efo {\end {eqnarray*} }
\def \ba {\begin {eqnarray*} }
\def \ea {\end {eqnarray*} }
\def \beq {\begin {eqnarray}}
\def \eeq {\end {eqnarray}}
\def \supp {\hbox{supp }}

\def \p {\partial}

\usepackage{color,comment}
\newcommand{\cF}{\mathcal{F}}
\newcommand{\cO}{\mathcal{O}}

\begin{document}
\title[Stability for biharmonic operator with first order perturbation]{Stability estimates for an inverse  boundary value problem for biharmonic operators with first order perturbation from partial data}

\author[Liu]{Boya Liu}
\address{B. Liu, Department of Mathematics\\
North Carolina State University\\ 
Raleigh, NC 27695, USA}
\email{bliu35@ncsu.edu}

\begin{abstract}
In this paper we study an inverse boundary value problem for the biharmonic operator with first order perturbation. Our geometric setting is that of a bounded simply connected domain in the Euclidean space of dimension three or higher. Assuming that the inaccessible portion of the boundary is flat, and we have knowledge of the Dirichlet-to-Neumann map on the complement, we prove logarithmic type stability estimates for both the first and the zeroth order perturbation of the biharmonic operator.
\end{abstract}
\maketitle

\section{Introduction and Statement of Results}

Let $\Omega \subseteq \R^n, n \geq 3$, be a bounded simply connected domain with smooth boundary $\p \Omega$. Consider a biharmonic operator with first order perturbation
\begin{equation}
\label{eq:def_biharmonic}
\mathcal{L}_{A,q} (x, D):= \Delta^2+A(x) \cdot D +q(x),
\end{equation}
where $D =i^{-1} \nabla$, $A$ is a complex-valued vector field, called the \textit{magnetic potential}, and $q$ is a complex-valued function, called the \textit{electric potential}. In practice, the operator $\mathcal{L}_{A,q}$ arises when we consider the equilibrium configuration of an elastic plate hinged along the boundary. In physics and geometry, higher order operators occur in the study of the Kirchhoff plate equation in the theory of elasticity, the continuum mechanics of buckling problems, and the study of the Paneitz-Branson operator in conformal geometry, see \cite{Ashbaugh,Gazzola_Grunau_Sweers,Meleshko}. We also refer readers to \cite{Campos,Nakamura_Uhlmann} for more on the elasticity model and perturbed biharmonic operators.

The operator $\mathcal{L}_{A,q}$ equipped with the domain
\[
\mathcal{D}(\mathcal{L}_{A,q}) = \{u \in H^4(\Omega): u|_{\partial \Omega} = (\Delta u)|_{\partial \Omega}= 0\}
\] 
is an unbounded closed operator on $L^2(\Omega)$ with purely discrete spectrum, see \cite[Chapter 11]{Grubb}. In this paper we shall make the following assumption about the eigenvalues of the operator $\mathcal{L}_{A,q}$.

\textbf{Assumption 1}: 
\label{asu:eigenvalue}
0 is not an eigenvalue of the perturbed biharmonic operator $	\mathcal{L}_{A,q}: \mathcal{D}(\mathcal{L}_{A,q}) \to L^2(\Omega)$.

The eigenvalues of $\mathcal{L}_{A,q}$ depend strongly on the coefficients $A$ and $q$. For instance, if the $L^\infty$-norms of $A$ and $q$ are strictly positive, then $\mathcal{L}_{A,q}$ satisfies Assumption \ref{asu:eigenvalue}.

Some properties of the eigenvalues and eigenfunctions of the polyharmonic operator $(-\Delta)^m$, $m\ge 2$, are discussed in \cite[Chapter 3]{Gazzola_Grunau_Sweers}. For the perturbed biharmonic operator $\mathcal{L}_{A,q}$, when $A=0$ and $q$ is real-valued and compactly supported in a ball of radius $R>0$,  the Dirichlet eigenvalues of the operator $\mathcal{L}_{0,q}$ are real-valued and bi-Lipschitz equivalent to $n^{4/3}$, see \cite[Lemma 3.6]{Li_Yao_Zhao}. To the best of our knowledge, the study of eigenvalue problems for $\mathcal{L}_{A,q}$ with $A\ne 0$ has not received any attention, and it is not within the scope of this paper.

Consider the boundary value problem with Navier boundary conditions
\begin{equation}
\label{eq:bvp}
\begin{cases}
\mathcal{L}_{A,q}u = 0 \quad \text{in} \quad \Omega, 
\\
u=f \quad \text{on} \quad \partial \Omega, 
\\
\Delta u=g \quad \text{on} \quad \partial \Omega,
\end{cases}
\end{equation}  
where $A \in W^{1, \infty}(\Omega, \C^n)$ and $q \in L^\infty(\Omega, \C)$.
Under Assumption \ref{asu:eigenvalue},  as explained in \cite{Krupchyk_Lassas_Uhlmann_bi_partial}, the boundary value problem \eqref{eq:bvp} has a unique solution $u\in H^4(\Omega)$ for any pair of functions $(f,g)\in H^{\frac{7}{2}}(\p \Omega)\times H^{\frac{3}{2}} (\p \Omega)$. Here and in what follows $H^s(\Omega)=\{U|_\Omega: U\in H^s(\R^n)\}$, $s\in \R$, is the standard $L^2$-based Sobolev space on $\Omega$, see \cite[Chapter 5]{Agranovich}. Also, throughout this paper we shall denote $H^{\alpha, \beta}(\partial \Omega)$ the product of two Sobolev spaces $H^\alpha(\partial \Omega)         \times H^\beta (\partial \Omega)$, equipped with the norm
\[
\|(f, g)\|_{H^{\alpha, \beta}(\partial \Omega)} = \|f\|_{H^\alpha(\partial \Omega)} + \|g\|_{H^\beta(\partial \Omega)}.
\]

Let us next describe the geometric framework considered in this paper. We assume that $\Omega \subseteq \{x=(x_1, \dots, x_n)\in \R^n: x_n<0\}$, $n\ge 3$. We split the boundary $\p \Omega$ in two non-empty subsets $\Gamma_0 := \p \Omega \cap \{x\in \R^n: x_n=0\}$, called the \textit{inaccessible part}, and its complement $\Gamma=\p \Omega \setminus \Gamma_0$. That is, the inaccessible part of the boundary is flat. This geometric setting was first considered in \cite{Isakov_07} in the context of inverse problems.

We define the partial Dirichlet-to-Neumann map associated with the boundary value problem \eqref{eq:bvp} on the open set $\Gamma \subset \p \Omega$
\color{black}
as
\begin{equation}
\label{eq:Cauchy_data_hyp}
\begin{aligned}
\Lambda_{A,q}^{\Gamma}: H^{\frac{7}{2}}(\Gamma)\times H^{\frac{3}{2}} (\Gamma)\to H^{\frac{5}{2}}(\Gamma)\times H^{\frac{1}{2}}(\Gamma), \quad \Lambda_{A,q}^{\Gamma}(f,g)=(\p_\nu u|_{\Gamma}, \p_\nu (\Delta u)|_{\Gamma}),
\end{aligned}
\end{equation}
where $\supp f, \: \supp g \subset \Gamma$, and $\nu$ is the outer unit normal of $\p \Omega$. Let us also define the norm of $\Lambda_{A, q}^\Gamma$ by
\[
\|\Lambda_{A, q}^\Gamma\| := \sup \{\|\Lambda_{A, q}^\Gamma(f, g)\|_{H^{\frac{5}{2},\frac{1}{2}}(\Gamma)}: \|(f, g)\|_{H^{\frac{7}{2}, \frac{3}{2}}(\Gamma)} = 1\}.
\]

The purpose of this paper is to establish log-type stability estimates for the magnetic potential $A$ and electric potential $q$ from the partial Dirichlet-to-Neumann map $\Lambda_{A,q}^{\Gamma}$. The corresponding uniqueness result was established in \cite[Theorem 1.4]{Yang}. 

The study of inverse boundary value problems concerning the first order perturbation of the polyharmonic operator $(-\Delta)^m$, $m\ge 2$, was initiated in \cite{Krupchyk_Lassas_Uhlmann_poly}, where  unique identification results from the Dirichlet-to-Neumann map were obtained. We highlight that unique recovery of the first order perturbation appearing in higher order elliptic operators differs from analogous problems for second order operators, such as the magnetic Schr\"odinger operator. In the latter case, due to the gauge invariance of boundary measurements, one can only hope to uniquely recover $dA$, which is the exterior derivative of the first order perturbation $A$, see for instance \cite{Krup_Lass_Uhl_magSchr_euc,Liu_2018,NakSunUlm_1995,Sun_95} and the references therein among the extensive literature in this direction. Here, if we view $A$ as a 1-form, then $dA$ is a 2-form given by the formula
\begin{equation}
\label{eq:def_mag_field}
dA=\sum_{1\le j<k\le n}(\p_{x_j}A_k-\p_{x_k}A_j)dx_j\wedge dx_k.
\end{equation} 
For the polyharmonic operator, which is of order $2m$, a gauge invariance occurs when we attempt to recover perturbations of order $2m-1$ from boundary measurements, see a very recent result \cite{Sahoo_Salo}.

Before stating the main results of this paper, let us introduce the following admissible sets for the coefficients $A$ and $q$ of the operator $\mathcal{L}_{A,q}$. Given a constant $M>0$, we define the class of admissible magnetic potentials $\mathcal{A}(\Omega, M)$ by 
\[
\mathcal{A}(\Omega,M)=\{A\in W^{1,\infty}(\Omega)\cap \mathcal{E}'(\overline{\Omega}): \|A\|_{H^s(\Omega)}\le M, s>\frac{n}{2}+1\}.
\]
Notice that we assumed \textit{a priori} bounds for the $H^s$-norm of $A$ instead of the natural $W^{1,\infty}$-norm in the definition above. We shall explain the reason in Subsection \ref{ssec:proof}.  Let us also define the class of admissible electric potentials $\mathcal{Q}(\Omega, M)$ in a similar fashion by 
\[
\mathcal{Q}(\Omega,M)=\{q\in L^\infty(\Omega): \|q\|_{L^\infty(\Omega)}\leq M\}.
\]

The main results of this paper are as follows. First, we state a log-type stability estimate for the magnetic potential $A$.
%
\begin{thm}
\label{thm:estimate_A_Linfty}
Let $\Omega \subseteq \{x\in \R^n: x_n<0\}$, $n\ge 3$, be a bounded simply connected domain with smooth boundary $\p \Omega$. Assume that $\Gamma_0 = \p \Omega \cap \{x\in \R^n: x_n=0\}$ is nonempty, and let $\Gamma=\p \Omega \setminus \Gamma_0$. Let $M \geq 0$, and let $A^{(j)}\in \mathcal{A}(\Omega, M)$ and $q_j\in \mathcal{Q}(\Omega, M)$, $j=1, 2$.  Suppose that Assumption \ref{asu:eigenvalue} holds for both operators $\mathcal{L}_{A^{(1)},q_1}$ and $\mathcal{L}_{A^{(2)},q_2}$. Then there exists a constant $C=C(\Omega, n, M,s)>0$ such that 
\begin{equation}
\label{eq:estimate_A_Linfty}
\|A^{(2)} - A^{(1)}\|_{L^\infty(\Omega)}
\leq 
C\left(\|\Lambda_{A^{(1)}, q_1}^\Gamma-\Lambda_{A^{(2)}, q_2}^\Gamma\|^{\frac{\eta^2}{2(1+s)^2}}+ \big|\log\|\Lambda_{A^{(1)}, q_1}^\Gamma - \Lambda_{A^{(2)}, q_2}^\Gamma\| \big|^{-\theta_1}\right),
\end{equation}
where $\eta=\frac{1}{2}(s-\frac{n}{2})$ and $0<\theta_1=\frac{\alpha^2 \eta^4}{(n+2)^2(1+s)^4}<1$.
\end{thm}

We shall also establish a log-type estimate for the electrical potential $q$, which is given in the following theorem.
\begin{thm}
\label{thm:estimate_q_hyp}
Under the same hypotheses as in Theorem \ref{thm:estimate_A_Linfty}, there exists a constant $C=C(\Omega, n, M,s)>0$ such that
\begin{equation}
\label{eq:estimate_q_hyp}
\|q_1 - q_2\|_{H^{-1}(\Omega)} \le
C\left(\|\Lambda_{A^{(1)}, q_1}^\Gamma-\Lambda_{A^{(2)}, q_2}^\Gamma\|^{\frac{\eta^2}{2(1+s)^2}}
+
\big|\log \|\Lambda_{A^{(1)}, q_1}^\Gamma-\Lambda_{A^{(2)}, q_2}^\Gamma\|\big|^{-\theta_2}\right),
\end{equation}
where $\eta=\frac{1}{2}(s-\frac{n}{2})$ and $0<\theta_2:=\frac{2\alpha^2\eta^4}{(n+2)^3(1+s)^4}<1$.
\end{thm}

If the electric potentials $q_1$ and $q_2$ pose additional regularity and a priori bounds, we have the following corollary of Theorem \ref{thm:estimate_q_hyp}.
\begin{cor}
\label{cor:estimate_q_Linfty}
Let $\Omega \subseteq \{x\in \R^n: x_n<0\}$, $n\ge 3$, be a bounded domain with smooth connected boundary. Assume that $\Gamma_0 := \p \Omega \cap \{x\in \R^n: x_n=0\}$ is nonempty, and let $\Gamma=\p \Omega \setminus \Gamma_0$. Let $M \geq 0$ and $s>\frac{n}{2}$. Let $A^{(j)}\in \mathcal{A}(\Omega, M)$, and let $q_j\in L^\infty(\Omega)$ be such that $\|q_j\|_{H^s(\Omega)}\le M$, $j=1, 2$.  Suppose that Assumption \ref{asu:eigenvalue} holds for both operators $\mathcal{L}_{A^{(1)},q_1}$ and $\mathcal{L}_{A^{(2)},q_2}$. Then there exists a constant $C=C(\Omega, n, M,s)>0$ such that
\begin{equation}
\label{eq:estimate_q_Linfty}
\|q_1 - q_2\|_{L^\infty(\Omega)} \le
C\bigg(\|\Lambda_{A^{(1)}, q_1}^\Gamma-\Lambda_{A^{(2)}, q_2}^\Gamma\|^{\frac{\eta^2}{2(1+s)^2}}
+
\big| \log  \|\Lambda_{A^{(1)}, q_1}^\Gamma-\Lambda_{A^{(2)}, q_2}^\Gamma\|\big|^{-\theta_2}\bigg)^{\frac{\eta}{1+s}}.
\end{equation}
Here the constants $\theta_2$ and $\eta$ are the same as in Theorem \ref{thm:estimate_q_hyp}.
\end{cor}

\begin{rem}
To the best of our knowledge, estimates \eqref{eq:estimate_A_Linfty} and \eqref{eq:estimate_q_hyp} are the first stability estimates for partial data inverse boundary value problems for the biharmonic operator with first order perturbations. 
\end{rem}


\subsection{Previous literature}

The study of inverse problems concerning biharmonic or polyharmonic operators have attracted significant attentions in recent years. For the case that only the zeroth order perturbation is considered, unique identifiability results were obtained in \cite{Ikehata,Isakov_91}. If we extend our consideration to include first order perturbations, it was proved in \cite{Krupchyk_Lassas_Uhlmann_poly} that the Dirichlet-to-Neumann map determines sufficiently smooth first order perturbation of the polyharmonic operator uniquely. There are extensive subsequent efforts to recover the first order perturbation of the polyharmonic operator with lower regularity, see for instance  \cite{Assylbekov_16,Assylbekov_Iyer,Brown_Gauthier,Krupchyk_Uhlmann_poly_16} and the references therein. We remark that relaxing the regularity of coefficients is crucial in inverse problems, since it enables the imaging of rough medium.  In addition to the aforementioned results in the Euclidean setting, we refer readers to \cite{Assylbekov_Yang,Yan} for some uniqueness results in the setting of Riemannian manifolds. 

In all of the aforementioned results, boundary measurements are made on the entire boundary of a domain or a manifold. However, in many applications such as geophysics, it could be either impossible or extremely cost-consuming to perform measurements on the entire boundary of a medium. Hence, it is of great interest and significance to investigate partial data inverse problems, in which case measurements are made only on open subsets of the boundary. In the context of inverse boundary value problems for biharmonic operators, it was established in \cite{Krupchyk_Lassas_Uhlmann_bi_partial} that the partial Dirichlet-to-Neumann map, where the Dirichlet data is measured on the entire boundary but the Neumann data is measured on slightly more than half of the boundary, determines the first and the zeroth order perturbations of the biharmonic operator uniquely. Several partial data uniqueness results were obtained in \cite{Yang} to recover the first and the zeroth order perturbation of the biharmonic operator from the set of Cauchy data, in bounded domains and infinite slabs.

Since polyharmonic operators are of order $2m$, it is natural to consider recovering second or higher order perturbations from boundary measurements, and some progress have  been made in this direction. For instance, the authors of \cite{Bhattacharyya_Ghosh_19,Bhattacharyya_Ghosh_second_order} obtained unique determination results for perturbations of up to second order appearing in the polyharmonic operator. Beyond the second order perturbations, it was proved in \cite{Bhattacharyya_Krishnan_Sahoo_21} that the Dirichlet-to-Neumann map, with the Neumann data measured on roughly half of the boundary, determines anisotropic perturbations of up to order $m$ appearing in polyharmonic operators.  Also, recovery of multiple isotropic perturbations, given as either a function or a vector field, of a polyharmonic operator from partial boundary measurements was established in \cite{Ghosh_Krishnan}. We would like to emphasize that, to the best of our knowledge, unique determination of the first and the zeroth order perturbation of the polyharmonic operator from partial boundary measurements, where both the Dirichlet data and the Neumann data are measured only on open subsets of the boundary, remain an important open problem. The analogous problem for the magnetic Schr\"odinger operator was solved in \cite{Chung_142}.

Turning attention to the issue of stability, the majority of known literature for polyharmonic operators considers only the zeroth order perturbation. A log-type estimate was obtained in \cite{Choudhury_Krishnan} when measurements are made on the entire boundary. In the realm of partial data results, a log-type estimate was proved in \cite{Choudhury_Heck} when the inaccessible part of the boundary is contained in a hyperplane. Additionally, a log-log-type estimate was established in \cite{Choudhury_Krishnan} in the scenario that the Neumann data is measured on slightly more than half of the boundary. We remark that the stability estimates probed in \cite{Choudhury_Heck,Choudhury_Krishnan} indicate that the inverse problems considered are severely ill-posed. This phenomenon means that small errors in measurements may result in exponentially large errors in the reconstruction of the unknown potential, which makes it very challenging to design reconstruction algorithms with high resolution in practice. However, as observed numerically in \cite{Colton_Haddar_Piana}, if a frequency $k$ is introduced to the operator, then stability may improve to a H\"older type or a Lipschitz type when $k$ becomes large. For the perturbed biharmonic operator $\mathcal{L}_{A,q}$, with $A=0$, the author derived a H\"older-type stability estimate for the potential $q$ at high frequencies from partial boundary measurements in \cite{Liu_stability}, where the inaccessible part of the boundary is flat.

Once stability estimates are established for the zeroth order perturbation of the polyharmonic operator, the natural next step is to derive estimates for higher order perturbations. To the best of our knowledge, stability estimates for higher order perturbations of the biharmonic operator were obtained only when full data is considered. We refer readers to see a very recent result \cite{Ma_Liu_stability} for log-type stability estimates for the first and zeroth perturbations of the biharmonic operator $\mathcal{L}_{A,q}$, as well as \cite{Aroua_Bellassoued} for stability estimates concerning perturbations of up the second order of polyharmonic operators.

\subsection{Outline of the proof of Theorem \ref{thm:estimate_A_Linfty} and Theorem \ref{thm:estimate_q_hyp}}
\label{ssec:proof}

The general strategy of the proof follows from the methods introduced in \cite{Alessandrini} using complex geometric optics (CGO) solutions. To elaborate, we shall construct a solution $v\in H^4(\Omega)$ to the equation $\mathcal{L}_{A^{(1)},q_1}^\ast v=0$ in $\Omega$, as well as a solution $u_2\in H^4(\Omega)$ satisfying the equation $\mathcal{L}_{A^{(2)},q_2} u_2=0$ in $\Omega$. Here $\mathcal{L}_{A,q}^\ast =\mathcal{L}_{\overline{A},q-i\div \overline{A}}$ is the formal $L^2$-adjoint of $\mathcal{L}_{A,q}$. Due to the restrictions $\supp u|_{\p \Omega}, \supp (\Delta u)|_{\p \Omega} \subseteq \Gamma$ from the definition of the partial Dirichlet-to-Neumann map \eqref{eq:Cauchy_data_hyp}, we need to construct solutions $v$ and $u_2$ such that $v|_{\Gamma_0}=(\Delta v)|_{\Gamma_0}=0$ and $u_2|_{\Gamma_0}=(\Delta u_2)|_{\Gamma_0}=0$. We shall accomplish this by employing the reflection argument originated in \cite{Isakov_07}. This approach has been successfully implemented to solve many partial data inverse boundary value problems where the inaccessible part of the boundary is contained in a hyperplane or a sphere. We refer readers to see for instance \cite{Caro_11,Choudhury_Heck,Krup_Lass_Uhl_magSchr_euc,Liu_stability,Potenciano_Machado_17,Yang} and the references therein. We would like to mention that we utilized three mutually orthogonal vectors in $\R^n$ satisfying  \eqref{eq:condition_vectors} to construct CGO solutions.  This is only possible when the domain $\Omega$ has dimension $n\ge 3$. Hence, the techniques used in this paper are not applicable if $n=2$, see \cite{Nachman_88} for the methods to establish a uniqueness result for the Calder\'on problem in dimension two.

Another important component of the proof is the integral inequality \eqref{eq:int_id_est_hyp}, whose derivation requires knowledge of the partial Dirichlet-to-Neumann map. This integral inequality, in conjunction with  CGO solutions \eqref{eq:CGO_v_explicit_hyp} and \eqref{eq:CGO_u2_explicit_hyp}, results in the Fourier transform of the difference of magnetic fields $dA^{(2)}-dA^{(1)}$, as well as some error terms. We then apply an estimate for the remainder terms of CGO solutions  \eqref{eq:est_remainder} and a quantitative version of the Riemann-Lebesgue lemma (Lemma \ref{lem:Riemann_Lebesgue}) to control those error terms. Subsequently, we use the definition of the $H^{-1}$-norm in terms of the Fourier transform, as well as judicious choice of various parameters, to deduce an estimate for $\|dA\|_{H^{-1}(\Omega)}$. 

The main difficulty to establish stability estimates for the magnetic potential $A$ is as follows. We can only obtain an estimate for the magnetic field $dA$ from the steps outlined above. However, as uniqueness results have been achieved for the magnetic potential $A$, we must also establish stability estimates for $A$ in a suitable norm. To overcome this difficulty we shall utilize a specific Hodge decomposition of vector fields \cite[Theorem 3.3.2]{Sharafutdinov}: For every vector field $A$, there exists a vector field $A^{\mathrm{sol}}$ and a function $\varphi$, both of which are uniquely determined, such that $A=A^{\mathrm{sol}}+\nabla \varphi$, where the \textit{solenoidal} part $A^{\mathrm{sol}}$ and the  \textit{potential} part $\nabla \varphi$ satisfy the properties $\div A^{\mathrm{sol}}=0$ and $\varphi|_{\p \Omega}=0$. Here $\div A^{\mathrm{sol}}$ is the divergence of $A^{\mathrm{sol}}$. Then an application of the Morrey's inequality yields
\[
\|A^{\mathrm{sol}}\|_{L^\infty(\Omega)}\le \|dA\|_{L^\infty(\Omega)}.
\]
Notice that we need an estimate for $\|dA\|_{L^\infty(\Omega)}$ in order to apply the Morrey's inequality, but we only established one for $\|dA\|_{H^{-1}(\Omega)}$ from earlier parts of the proof. To medicate this, as we shall detail in estimate \eqref{eq:est_dA_Linfty}, we apply the Sobolev embedding theorem and the interpolation inequality \cite[Theorem 7.22]{Grubb} to obtain an estimate for $\|dA\|_{L^\infty(\Omega)}$. We point out that this is the reason why we assumed an a priori bound for the norm $\|A\|_{H^s(\Omega)}$, $s>\frac{n}{2}+1$, in the definition of  $\mathcal{A}(\Omega, M)$. This is also the reason why we have a log-type estimate for $\|A\|_{L^\infty(\Omega)}$ instead of  $\|A\|_{H^{-1}(\Omega)}$ in Theorem \ref{thm:estimate_A_Linfty}.

To prove an estimate for $\|A\|_{L^\infty(\Omega)}$, we still need to establish an estimate for $\|\nabla \varphi\|_{L^\infty(\Omega)}$. This is accomplished in Lemma \ref{lem:est_test_function_gradient} by utilizing the CGO solutions \eqref{eq:CGO_v_explicit_hyp} and \eqref{eq:CGO_u2_explicit_hyp} again. From here, estimate \eqref{eq:estimate_A_Linfty} follows immediately by combining the estimates for $\|A^{\mathrm{sol}}\|_{L^\infty(\Omega)}$ and $\|\nabla \varphi\|_{L^\infty(\Omega)}$.  

Finally, we would like to point out that the techniques utilized to prove the main results in this paper do not enable us to establish estimates \eqref{eq:estimate_A_Linfty} and \eqref{eq:estimate_q_hyp} simultaneously. Due to the $Du_2$ term in the integral inequality \eqref{eq:int_id_est_hyp}, there will be a term of magnitude $h^{-1}$ after the CGO solutions  \eqref{eq:CGO_v_explicit_hyp} and \eqref{eq:CGO_u2_explicit_hyp} are substituted into the left-hand side of \eqref{eq:int_id_est_hyp}.  To medicate this issue, we need to multiply the left-hand side of \eqref{eq:int_id_est_hyp} by $h$. This step is done in the proof of Proposition \ref{prop:est_dA_H-1norm}. It then follows from estimate \eqref{eq:est_qterm} that the term involving the electric potential $q$ on the left-hand side of \eqref{eq:int_id_est_hyp} vanishes as $h\to 0$, which prevents us from achieving simultaneous reconstruction of both potentials $A$ and $q$.

This paper is organized as follows. In Section \ref{sec:prelim_results} we state a result pertaining the existence of CGO solutions, as well as a quantitative version of the Riemann-Lebesgue lemma, which will be applied in the proof. Section \ref{sec:proof_mag_term} is devoted to proving Theorem \ref{thm:estimate_A_Linfty}, the estimate for the magnetic potential. Finally, we verify the estimates concerning the electric potential, i.e.,  Theorem \ref{thm:estimate_q_hyp} and Corollary \ref{cor:estimate_q_Linfty}, in Section \ref{sec:proof_elec_term}.

\subsection*{Acknowledgments}
We would like to express our gratitude to the anonymous referees for their valuable feedback, which led to significant improvements of the paper.

\section{Preliminary Results}
\label{sec:prelim_results}
In this short section we collect some results that are necessary to establish the stability estimates in this paper. Let us start by stating a result concerning the existence of CGO solutions to the equation $\mathcal{L}_{A,q}u=0$ in $\Omega$ of the form
\[
u(x,\zeta, h)=e^{\frac{x\cdot \zeta}{h}}(a(x,\zeta)+r(x,\zeta, h)).
\] 
Here $\zeta \in \C^n$ is a complex vector satisfying $\zeta\cdot \zeta=0$, $h>0$ is a small semiclassical parameter, $a$ is a smooth amplitude, and $r$ is a correction term that vanishes  in a suitable sense in the limit $h\to 0$. We refer readers to see \cite{Krupchyk_Lassas_Uhlmann_bi_partial,Krupchyk_Lassas_Uhlmann_poly,Yang} for detailed discussions and proofs. In the following proposition, we equip the domain $\Omega$ with a semiclassical Sobolev norm
\[
\|u\|^2_{H^1_\scl(\Omega)}=\|u\|^2_{L^2(\Omega)}+\|hD u\|^2_{L^2(\Omega)}.
\]
\begin{prop}
\label{prop:CGO_solutions}
Let $A \in W^{1,\infty}(\Omega, \C^n)$ and $q \in L^\infty(\Omega, \C)$. Let $\zeta \in \C^n$ be a complex vector such that $\zeta \cdot \zeta = 0, \: \zeta = \zeta^{(0)}+\zeta^{(1)}$, where $\zeta^{(0)}$ is independent of $h>0$, $|\Re \zeta^{(0)}| = |\Im \zeta^{(0)}|=1$, and $\zeta^{(1)} = \mathcal{O}(h)$ as $h \to 0$. Then for all $h>0$ small enough, there exists a solution $u \in H^4(\Omega)$ to the equation $\mathcal{L}_{A,q}u=0$ in $\Omega$ of the form 
\begin{equation}
	\label{eq:CGO_form_general}
u(x, \zeta, h) = e^{\frac{ix \cdot \zeta}{h}}(a(x, \zeta^{(0)})+r(x, \zeta, h)),
\end{equation}
where $a \in C^\infty(\overline{\Omega})$ satisfies the transport equation
\begin{equation}
\label{eq:transport_equation}
(\zeta^{(0)} \cdot \nabla)^2a = 0 \quad \text{in} \quad \Omega,
\end{equation}
and $\|r\|_{H^1_{\scl}(\Omega)} =\cO(h)$ as $h\to 0$.
\end{prop}

We also need the following quantitative version of the Riemann-Lebesgue lemma to control the error terms involving Fourier transforms. This result was originally proved in \cite[Lemma 2.4]{Choudhury_Heck}.
\begin{lem}
\label{lem:Riemann_Lebesgue} 
Let $f\in L^1(\R^n)$ and suppose that there exist constants $\delta>0$, $C_0>0$ and $\alpha \in (0,1)$ such that for $|y|<\delta$, we have
\begin{equation}
\label{eq:smallness}
\|f(\cdot-y)-f(\cdot)\|_{L^1(\R^n)}\le C_0|y|^\alpha.
\end{equation}
Then there exist constant $C=C(C_0, \|f\|_{L^1}, n,\delta,\alpha)>0$ and $\varepsilon_0>0$ such that the inequality
\begin{equation}
\label{eq:Riemann_Lebesgue}
|\cF (f)(\xi)|\le C(e^{-\frac{\varepsilon^2|\xi|^2}{4\pi}}+\varepsilon^\alpha)
\end{equation}
holds for any $0<\varepsilon<\varepsilon_0$.
\end{lem}

\section{Estimate for the Magnetic Potential}
\label{sec:proof_mag_term}

The goal of this section is to prove Theorem \ref{thm:estimate_A_Linfty}. We shall start this section by constructing CGO solutions vanishing on the inaccessible part $\Gamma_0$ of the boundary, followed by establishing an estimate for the magnetic field $dA$. Finally, we utilize a special Hodge decomposition to verify estimate \eqref{eq:estimate_A_Linfty}.

\subsection{Construction of CGO solutions}
In this subsection we utilize the reflection arguments originally developed in \cite{Isakov_07} to construct CGO solutions $u_2$ and $v$ to the equations $\mathcal{L}_{A^{(2)}, q_2}u_2=0$ and $\mathcal{L}_{A^{(1)}, q_1}^\ast v=0$ in $\Omega$, respectively. In particular, the solutions $u_2$ and $v$ satisfy
\begin{equation}
\label{eq:condition_hyp}
u_2|_{\Gamma_0}=(\Delta u_2)|_{\Gamma_0}=0, \quad v|_{\Gamma_0}=(\Delta v)|_{\Gamma_0}=0.
\end{equation} 

Let us first introduce some notations. In what follows we shall denote $x^\ast=(x', -x_n)$ for all $x=(x',x_n)\in \R^{n-1}\times \R$, and denote $f^\ast(x)=f(x^\ast)$ for any function $f$.

We shall proceed to construct CGO solutions as follows. For a given non-zero vector $\xi\in \R^n$, we choose unit vectors $\mu^{(1)}, \mu^{(2)}\in \R^n$ in an appropriate way, which we shall specify later. The vectors $\xi$, $\mu^{(1)}$, and $\mu^{(2)}$ will be used to construct CGO solutions. Let us write $\xi=(\xi',\xi_n)$, where $\xi'=(\xi_1,\dots, \xi_{n-1})\in \R^{n-1}$, and assume first that $\xi' \ne 0$. We then define vectors $e(1)=(\frac{\xi'}{|\xi'|},0)$, $e(n)=(0,\dots, 0, 1)$, and complete $e(1), \dots, e(n)$ to an orthonormal basis in $\R^n$, which we shall denote by $\{e(1), e(2), \dots, e(n)\}$. In this basis, the vector $\xi$ has the coordinate representation $\tilde \xi=(|\xi'|,0,\dots,0,\xi_n )$. If $\xi'=0$, we let $\{e(1), \dots, e(n)\}$ be the standard basis in $\R^n$.

For any pair of vectors $\eta^{(1)},\eta^{(2)}\in \R^n$, we denote by $\tilde \eta^{(1)}, \tilde \eta^{(2)}$ their coordinate representations in the  basis $\{e(1), \dots, e(n)\}$. Then we have 
\[
\eta^{(1)}\cdot \eta^{(2)}=\tilde \eta^{(1)}\cdot \tilde \eta^{(2)}, \quad \eta^{(1)}_n=\tilde \eta^{(1)}_n,\quad \eta^{(2)}_n=\tilde \eta^{(2)}_n.
\]
Let $\mu^{(1)}, \mu^{(2)}$ be vectors in $\R^n$ such that
\begin{equation}
\label{eq:def_mu12_new_basis}
\tilde \mu^{(1)}=\bigg(-\frac{\xi_n}{|\xi|}, 0,\dots, 0, \frac{|\xi'|}{|\xi|}\bigg) , \quad \tilde \mu^{(2)}=(0,1,0,\dots, 0).
\end{equation}
Then direct computations show that 
\begin{equation}
\label{eq:condition_vectors}
\mu^{(1)}\cdot\mu^{(2)}=\mu^{(1)}\cdot \xi=\mu^{(2)}\cdot \xi=0, \quad |\mu^{(1)}|=|\mu^{(2)}|=1.
\end{equation}

We now set
\begin{equation}
\label{eq:form_zeta}
\begin{aligned}
\zeta_1=\frac{h\xi}{2}+\sqrt{1-h^2\frac{|\xi|^2}{4}}\mu^{(1)}+i \mu^{(2)}, \quad
\zeta_2=-\frac{h\xi}{2}+\sqrt{1-h^2\frac{|\xi|^2}{4}} \mu^{(1)}-i \mu^{(2)},
\end{aligned}
\end{equation}
where $0<h\ll 1$ is a small semiclassical parameter such that $1-h^2 \frac{|\xi|^2}{4} \ge 0$. It  follows from direct computations that $\zeta_j\cdot\zeta_j=0$, $j=1,2$, and $\frac{\zeta_2-\overline{\zeta_1}}{h}=-\xi$. Moreover, we see that $\zeta_1 = \mu^{(1)} + i\mu^{(2)} +\mathcal{O}(h)$ and $\zeta_2 = \mu^{(1)} - i\mu^{(2)} + \mathcal{O}(h)$ as $h \to 0$. From the definition of $\zeta^{(0)}$ given in Proposition \ref{prop:CGO_solutions}, we have
\[
\zeta_1^{(0)}= \mu^{(1)} + i\mu^{(2)} \quad \text{ and } \quad \zeta_2^{(0)}= \mu^{(1)} - i\mu^{(2)}. 
\]

In order to fulfill the conditions in \eqref{eq:condition_hyp}, similar to \cite{Isakov_07,Krup_Lass_Uhl_magSchr_euc,Potenciano_Machado_17,Yang}, we reflect $\Omega$ with respect to the hyperplane $\{x_n=0\}$ and denote this reflection by 
\[
\Omega^\ast :=\{x^\ast: x=(x', x_n)\in \Omega\}, \quad x'=(x_1, \dots, x_{n-1})\in \R^{n-1}.
\]
We also reflect the magnetic potential $A=(A_1, \dots, A_n)$ and the electric potential $q$ with respect to $\{x_n=0\}$. Specifically, we shall make even extensions for $A_j$, $j=1, \dots, n-1$, and make an odd extension for $A_n$. That is, we set
\begin{equation}
\label{eq:extension_A}	
\tilde{A}_j(x)=
\begin{cases}
A_j(x), \: x_n<0,
\\
A_j(x^\ast), \:  x_n > 0,
\end{cases}
j=1, \dots, n-1, \quad 
\tilde{A}_n(x)=
\begin{cases}
A_n(x), \: x_n<0,
\\
-A_n(x^\ast), \:  x_n > 0.
\end{cases}
\end{equation}
In the same way, we extend electric potential $q$ evenly to write 
\begin{equation}
\label{eq:extension_q}
\tilde{q}(x)=
\begin{cases}
q(x), \quad x_n<0,
\\
q(x^\ast), \quad x_n > 0.
\end{cases}
\end{equation}
As explained in \cite[Section 5]{Yang}, we have $\tilde A^{(j)}\in W^{1,\infty}(\Omega \cup \Omega^\ast;\C^n)$ and $q_j\in L^\infty(\Omega \cup \Omega^\ast;\C)$, $j=1,2$. Furthermore, we get $\tilde{A}^{(j)}\in \mathcal{A}(\Omega \cup \Omega^\ast, 2M)$ and $q_j\in \mathcal{Q}(\Omega \cup \Omega^\ast, 2M)$.

Let $B=B(0,\tilde{R})$ be a ball in $\R^n$ centered at the origin with radius $\tilde R\ge 1$ such that $\Omega \cup \Omega^\ast\subset \subset B$.  By Proposition \ref{prop:CGO_solutions}, there exists a solution $v\in H^4(B)$ to the equation $\mathcal{L}_{\tilde A^{(1)}, \tilde q_1}^\ast \tilde v=0$ in $B$ of the form
\begin{equation}
\label{eq:u1_form}
\tilde v(x, \zeta_1, h)=e^{\frac{ix\cdot \zeta_1}{h}}(a_1(x,  \mu^{(1)} + i\mu^{(2)})+r_1(x, \zeta_1,h)).
\end{equation}
There also exists a solution $u_2\in H^4(B)$ to the equation $\mathcal{L}_{\tilde A^{(2)}, \tilde q_2} \tilde{u_2}=0$ in $B$ given by
\begin{equation}
\label{eq:u2_form}
\tilde{u_2}(x, \zeta_2, h)=e^{\frac{ix\cdot \zeta_2}{h}}(a_2(x, \mu^{(1)}-i\mu^{(2)})+r_2(x, \zeta_2,h)).
\end{equation}
Here the amplitudes $a_j\in C^\infty(\overline{B})$, $j=1,2$, solve the transport equations 
\begin{equation}
\label{eq:trans_eq_u1_u2}
((\mu^{(1)} + i\mu^{(2)}) \cdot \nabla)^2a_1(x, \mu^{(1)} + i\mu^{(2)}) = 0 \: \text{ and }((\mu^{(1)} - i\mu^{(2)}) \cdot \nabla)^2a_2(x, \mu^{(1)} - i\mu^{(2)})=0 \: \text{ in } B, 
\end{equation}
respectively, and the correction terms $r_j\in H^4(B)$, $j=1,2$, satisfy the estimate 
\begin{equation}
\label{eq:est_remainder}
\|r_j\|_{H^1_{\scl}(B)}=\cO(h), \quad h\to 0.
\end{equation}
By the interior elliptic regularity, we have $\tilde{v}, \tilde{u_2}\in H^4(\Omega\cup \Omega^\ast)$ and
\begin{equation}
\label{eq:est_r_domain}
\|r_j\|_{H^1_{\scl}(\Omega \cup \Omega^\ast)}=\cO(h), \quad h\to 0.
\end{equation}

We now proceed to construct CGO solutions in $\Omega$ satisfying the conditions in \eqref{eq:condition_hyp}. To that end, we set
\[
v(x)=\tilde v(x)-\tilde{v^\ast}(x) \quad \text{and} \quad  u_2(x)=\tilde{u_2}(x)-\tilde{u_2^\ast}(x),
\]
where $\tilde v$ and $\tilde{u_2}$ are CGO solutions given by \eqref{eq:u1_form} and \eqref{eq:u2_form}, respectively. Then we have $v,u_2\in H^4(\Omega)$. Furthermore, direct computations yield that $v$ and $u_2$ satisfy the conditions in \eqref{eq:condition_hyp}. 
%
Hence, the explicit forms of the desired CGO solutions $v$ and $u_2$ in $\Omega$ are given by
\begin{equation}
\label{eq:CGO_v_explicit_hyp}
v(x, \zeta_1, h)=e^{\frac{ix\cdot \zeta_1}{h}}(a_1(x, \zeta^{(0)}_1)+r_1(x, \zeta_1,h))-e^{\frac{i x \cdot \zeta_1^\ast}{h}}(a_1^\ast (x, \zeta^{(0)}_1)+r_1^\ast(x, \zeta_1,h))
\end{equation}
and
\begin{equation}
\label{eq:CGO_u2_explicit_hyp}
u_2(x, \zeta_2, h)=e^{\frac{ix\cdot \zeta_2}{h}}(a_2(x, \zeta^{(0)}_2)+r_2(x, \zeta_2,h))-e^{\frac{ix \cdot \zeta_2^\ast}{h}}(a_2^\ast(x, \zeta^{(0)}_2)+r_2^\ast(x, \zeta_2,h)).
\end{equation}

\subsection{Derivation of an estimate for the magnetic field}
The goal of this subsection is to utilize the CGO solutions \eqref{eq:CGO_v_explicit_hyp} and \eqref{eq:CGO_u2_explicit_hyp} to derive an estimate for the magnetic field $d(A^{(2)}-A^{(1)})$. Our starting point is the following integral inequality.

\begin{lem}
\label{lem:est_int_id}
Let $A^{(j)}\in W^{1,\infty}(\Omega, \C^n)$, $q_j\in L^\infty(\Omega, \C)$, and let  $\Lambda_{A^{(j)}, q_j}^\Gamma$ be the partial Dirichlet-to-Neumann map associated with the operator $\mathcal{L}_{A^{(j)}, q_j}$, $j=1,2$. Then there exists a constant $C>0$ such that the estimate 
\begin{equation}
\label{eq:int_id_est_hyp}
\bigg|\int_\Omega [(A^{(2)}-A^{(1)})\cdot Du_2+(q_2-q_1)u_2]\overline{v}dx\bigg|\le Ce^{\frac{9R}{h}}\|\Lambda_{A^{(1)}, q_1}^\Gamma-\Lambda_{A^{(2)}, q_2}^\Gamma\|
\end{equation}
holds for any functions $u_2, v\in H^4(\Omega)$ that satisfy the equations $\mathcal{L}_{A^{(2)}, q_2}u_2=0$ and $\mathcal{L}_{A^{(1)}, q_1}^\ast v=0$ in $\Omega$.
\end{lem}

\begin{proof}
We begin the proof by recalling the Green's formula, which holds for all $u,v\in H^4(\Omega)$, see \cite{Grubb}:
\begin{equation}
\label{eq:Green_id}
\begin{aligned}
\int_\Omega(\mathcal{L}_{A,q}u)\overline{v}dx-\int_{\Omega}u\overline{\mathcal{L}_{A,q}^\ast v}dx = 
&-i\int_{\p \Omega}\nu \cdot uA\overline{v}dS+\int_{\p \Omega}\p_\nu (\Delta u)\overline{v}dS-\int_{\p \Omega}(\Delta u)\overline{\p_\nu v}dS
\\
&+\int_{\p \Omega}\p_\nu u(\overline{\Delta v})dS-\int_{\p \Omega}u\overline{\p_\nu(\Delta v)}dS.
\end{aligned}
\end{equation}
Here $\mathcal{L}_{A,q}^\ast$ is the $L^2$-adjoint of $\mathcal{L}_{A,q}$, and $\nu$ is the outward unit normal to the boundary $\p \Omega$.

Let $u_j$, $j=1,2$, be solutions to the boundary value problem \eqref{eq:bvp} with coefficients $A^{(j)}, q_j$,  respectively. By a direct computation, we see that the function $u:=u_1-u_2$ solves the following equation
\begin{equation}
\label{eq:difference}
\begin{cases}
\mathcal{L}_{A^{(1)}, q_1}u=(A^{(2)}-A^{(1)})\cdot Du_2+(q_2-q_1)u_2 \quad  \text{ in } \quad  \Omega,
\\
u=0 \quad \text{ on } \quad  \p \Omega,
\\
\Delta u=0 \quad \text{ on } \quad  \p \Omega.
\end{cases}
\end{equation}
Multiplying both sides of the first equation in \eqref{eq:difference} by $\overline{v}$ and applying the Green's formula \eqref{eq:Green_id}, we obtain the integral identity
\begin{equation}
\label{eq:int_id_A_hyp}
\int_\Omega [(A^{(2)}-A^{(1)})\cdot Du_2+(q_2-q_1)u_2]\overline{v}dx
=
\int_{\Gamma}\p_\nu (\Delta (u_1-u_2))\overline{v}dS
+\int_{\Gamma}\p_\nu (u_1-u_2)(\overline{\Delta v})dS,
\end{equation}
where we have also used $\supp v, \supp(\Delta v)\subseteq \Gamma$.

Let us now estimate the right-hand side of \eqref{eq:int_id_A_hyp}. To that end, we apply the Cauchy-Schwartz inequality and the trace theorem to deduce
\begin{equation}
\label{eq:int_id_rhs_est_hyp}
\begin{aligned}
\bigg|& \int_{\Gamma}\p_\nu (\Delta (u_1-u_2))\overline{v}dS
+\int_{\Gamma}\p_\nu (u_1-u_2)(\overline{\Delta v})dS\bigg|
\\
&\le \|\p_\nu (\Delta (u_1-u_2))\|_{L^2(\Gamma)} \|v\|_{L^2(\Gamma)}+ \|\p_\nu (u_1-u_2)\|_{L^2(\Gamma)} \|\Delta v\|_{L^2(\Gamma)}
\\
&\le C ( \|\p_\nu (\Delta (u_1-u_2))\|_{L^2(\Gamma)} \|v\|_{H^1(\Omega)}+ \|\p_\nu (u_1-u_2)\|_{L^2(\Gamma)} \|\Delta v\|_{H^1(\Omega)})
\\
&\le C(\|\p_\nu (\Delta (u_1-u_2))\|_{H^{\frac{1}{2}}(\Gamma)}+\|\p_\nu (u_1-u_2)\|_{H^{\frac{5}{2}}(\Gamma)})(\|v\|_{H^1(\Omega)}+\|\Delta v\|_{H^1(\Omega)})
\\
&\le C\|(\Lambda_{A^{(1)}, q_1}^\Gamma-\Lambda_{A^{(2)}, q_2}^\Gamma)(f,g)\|_{H^{\frac{5}{2}, \frac{1}{2}}(\Gamma)}(\|v\|_{H^1(\Omega)}+\|\Delta v\|_{H^1(\Omega)})
\\
&\le C\|\Lambda_{A^{(1)}, q_1}^\Gamma-\Lambda_{A^{(2)}, q_2}^\Gamma\| \|(f,g)\|_{H^{\frac{7}{2}, \frac{3}{2}}(\Gamma)}(\|v\|_{H^1(\Omega)}+\|\Delta v\|_{H^1(\Omega)})
\\
&\le C\|\Lambda_{A^{(1)}, q_1}^\Gamma-\Lambda_{A^{(2)}, q_2}^\Gamma\| (\|u_2\|_{H^4(\Omega)}+\|\Delta u_2\|_{H^2(\Omega)})(\|v\|_{H^1(\Omega)}+\|\Delta v\|_{H^1(\Omega)}).
\end{aligned}
\end{equation}

To proceed, we need to estimate the $H^s$-norms of $u_2, v$, and their derivatives appearing in the inequality above. Since $\Omega$ is a bounded domain, there exists a constant $R>0$ such that $\Omega \subseteq B(0, R)$. Thus, we have the inequality $\big|e^{\frac{ix\cdot \zeta_j}{h}}\big|\le e^{\frac{2R}{h}}$, where we have utilized the fact that $|\zeta_j|\le 2$. By the same computations as in \cite[Section 4]{Choudhury_Krishnan}, we obtain the estimates
\begin{equation}
\label{eq:est_solutions_hyp}
\|v\|_{H^1(\Omega)}\le \frac{C}{h}e^{\frac{2R}{h}}, \quad \|\Delta v\|_{H^1(\Omega)}\le \frac{C}{h}e^{\frac{2R}{h}}, \quad \|u\|_{H^4(\Omega)}\le \frac{C}{h^4}e^{\frac{2R}{h}}, \quad \|\Delta u_2\|_{H^2(\Omega)}\le \frac{C}{h}e^{\frac{2R}{h}}.
\end{equation}
Therefore, \eqref{eq:int_id_rhs_est_hyp} and \eqref{eq:est_solutions_hyp} yield
\begin{align*}
\bigg| &\int_{\Gamma}\p_\nu (\Delta (u_1-u_2))\overline{v}dS
+\int_{\Gamma}\p_\nu (u_1-u_2)(\overline{\Delta v})dS\bigg| 
\\
&\le C\|\Lambda_{A^{(1)}, q_1}-\Lambda_{A^{(2)}, q_2}\| \big( \frac{C}{h}e^{\frac{2R}{h}}+ \frac{C}{h^4}e^{\frac{2R}{h}}\big) \big( \frac{C}{h}e^{\frac{2R}{h}}+ \frac{C}{h}e^{\frac{2R}{h}}\big)
\\
&\le \frac{C}{h^5}e^{\frac{4R}{h}}\|\Lambda_{A^{(1)}, q_1}-\Lambda_{A^{(2)}, q_2}\|.
\end{align*}

From here, we obtain the claimed estimate \eqref{eq:int_id_est_hyp} from the inequality above by using the fact that $\frac{1}{h}\le e^{\frac{R}{h}}$. This completes the proof of Lemma \ref{lem:est_int_id}.
\end{proof}

We are now ready to state and prove an estimate for magnetic field.
\begin{prop}
\label{prop:est_dA_H-1norm}
Let $\Omega \subseteq \{x\in \R^n: x_n<0\}$, $n\ge 3$, be a bounded domain with smooth connected boundary $\p \Omega$. Suppose that $\Gamma_0=\p \Omega\cap \{x\in \R^n: x_n=0\}$ is non-empty, and let $\Gamma = \p \Omega \setminus \Gamma_0$. Let $A^{(j)}\in \mathcal{A}(\Omega, M)$ and $q_j\in \mathcal{Q}(\Omega, M)$, $j=1,2$. Then there exists constants $C>0$ and $h_0>0$ such that the estimate
\begin{equation}
\label{eq:H-1_est_dA}
\|d(A^{(2)}-A^{(1)})\|_{H^{-1}(\Omega)}\le Ce^{\frac{10R}{h}}\|\Lambda_{A^{(1)}, q_1}^\Gamma-\Lambda_{A^{(2)}, q_2}^\Gamma\|+Ch^{\frac{\alpha}{n+2}}
\end{equation} 
holds for $0<h<h_0$ and some $\alpha \in (0,1)$. Here $R>0$ is a constant such that $\Omega \subseteq B(0,R)$.
\end{prop}

\begin{proof}
We shall derive \eqref{eq:H-1_est_dA} by estimating the Fourier transform of $d(A^{(2)}-A^{(1)})$. To that end, we extend $A^{(j)}$ and $q_j$, $j=1,2$, by zero to $\R^n\setminus \Omega$ and denote the extensions by the same letters. It follows from \cite[Chapter 1, Theorem 3.4.4]{Agranovich} that $A^{(j)}\in W^{1,\infty}(\R^n;\C^n)$. For convenience, we shall denote $A:=A^{(2)}-A^{(1)}$ and $q:=q_2-q_1$ for the remainder of this section.

We now substitute the CGO solutions $v$ and $u_2$, given by \eqref{eq:CGO_v_explicit_hyp} and \eqref{eq:CGO_u2_explicit_hyp}, respectively, into the left-hand side of the integral inequality \eqref{eq:int_id_est_hyp}, multiply the resulted expression by $h$, and let $h\to 0$. This requires us to compute the quantities $\overline{v}Du_2$ and $\overline{v}u_2$, which contain expressions of the form
\begin{equation}
\label{eq:sum_diff_exp}
e^{\frac{ix\cdot(\zeta_2-\overline{\zeta_1})}{h}}=e^{-ix\cdot \xi}, \quad e^{\frac{ix\cdot(\zeta_2^\ast-\overline{\zeta_1^\ast})}{h}}=e^{-ix^\ast\cdot \xi}, \quad 
e^{\frac{ix\cdot (\zeta_2^\ast- \overline{\zeta_1})}{h}}=e^{-ix\cdot \xi_+}, \quad 
e^{\frac{ix\cdot (\zeta_2- \overline{\zeta_1^\ast})}{h}}=e^{-ix\cdot \xi_-},
\end{equation}
where 
\begin{equation}
\label{eq:def_pm_xi}
\xi_\pm = \bigg(\xi', \pm \frac{2}{h}\sqrt{1-h^2\frac{|\xi|^2}{4}}\frac{|\xi'|}{|\xi|}\bigg).
\end{equation}
Hence, we write the first term on the left-hand side of \eqref{eq:int_id_est_hyp} as
\begin{equation}
\label{eq:int_id_substitute_A}
\begin{aligned}
&h\int_\Omega A\cdot Du_2\overline{v}dx
\\
= &\int_\Omega A\cdot i(\zeta_2 e^{-ix\cdot \xi}\overline{a_1}a_2+ \zeta_2^\ast e^{-ix^\ast\cdot \xi}\overline{a_1^\ast}a_2^\ast)dx 
+ \int_\Omega A\cdot i(-\zeta_2 e^{-ix\cdot \xi_-}\overline{a_1^\ast}a_2- \zeta_2^\ast e^{-ix\cdot \xi_+}\overline{a_1}a_2^\ast)dx
\\
&+\int_\Omega A\cdot w_1 dx
\\
:=&I_1+I_2+I_3,
\end{aligned}
\end{equation}
where
\begin{align*}
w_1 &= i\zeta_2 \big[e^{-ix\cdot \xi}(\overline{a_1}r_2+a_2\overline{r_1}+\overline{r_1}r_2)
-e^{-ix\cdot \xi_-}(\overline{a_1^\ast}r_2+a_2\overline{r_1^\ast}+\overline{r_1^\ast}r_2)\big]
\\
&+i\zeta_2^\ast\big[ -e^{-ix\cdot \xi_+}(\overline{a_1}r_2^\ast+a_2^\ast\overline{r_1}+\overline{r_1}r_2^\ast)+e^{-ix^\ast\cdot \xi}(\overline{a_1^\ast}r_2^\ast+\overline{r_1^\ast}a_2^\ast+\overline{r_1^\ast}r_2^\ast)\big]
\\
&+h\big[e^{-ix\cdot \xi}(\overline{a_1}Da_2+\overline{a_1}Dr_2+\overline{r_1}Da_2+\overline{r_1}Dr_2)-e^{-ix\cdot \xi_-}(\overline{a_1^\ast}Da_2+\overline{a_1^\ast}Dr_2+\overline{r_1^\ast}Da_2+\overline{r_1^\ast}Dr_2)
\\
& -e^{-ix\cdot \xi_+} (\overline{a_1}Da_2^\ast+\overline{a_1}Dr_2^\ast+\overline{r_1}Da_2^\ast+\overline{r_1}Dr_2^\ast)
+e^{-ix^\ast\cdot \xi}(\overline{a_1^\ast}Da_2^\ast
+\overline{a_1^\ast}Dr_2^\ast+\overline{r_1^\ast}Da_2^\ast+\overline{r_1^\ast}Dr_2^\ast)
\big].
\end{align*}

We now investigate the limit of each integral in \eqref{eq:int_id_substitute_A} as $h\to 0$. To analyze $I_1$, we recall the form of $\zeta_2$ in \eqref{eq:form_zeta} and the definition of $\tilde{A}$ from \eqref{eq:extension_A}, in conjunction with making a change of variables, to deduce that
\begin{equation}
\label{eq:I1_sum}
I_1=i\zeta_2\cdot \int_{\Omega \cup \Omega^\ast} \tilde{A} e^{-ix\cdot \xi}\overline{a_1}a_2dx
\to i(\mu^{(1)}-i\mu^{(2)})\cdot \int_{\Omega \cup \Omega^\ast}  \tilde{A}e^{-ix\cdot \xi}\overline{a_1}a_2dx, \quad h\to 0.
\end{equation}

By replacing the vector $\mu^{(2)}$ in \eqref{eq:I1_sum} by $-\mu^{(2)}$, we have
\begin{equation}
\label{eq:I1_difference}
I_1\to i(\mu^{(1)}+i\mu^{(2)})\cdot \int_{\Omega \cup \Omega^\ast}  \tilde A e^{-ix\cdot \xi}\overline{a_1}a_2dx, \quad h\to 0.
\end{equation}
Hence, in view of \eqref{eq:I1_sum} and \eqref{eq:I1_difference}, we see that
\begin{equation}
\label{eq:I1_Fourier_a1a2}
I_1\to i\mu \cdot \int_{\Omega \cup \Omega^\ast}   \tilde A e^{-ix\cdot \xi}\overline{a_1}a_2dx, \quad h\to 0,
\end{equation}
for any $\mu \in \mathrm{Span}\{\mu^{(1)}, \mu^{(2)}\}$ and all $\xi \in \R^n$ satisfying \eqref{eq:condition_vectors}. By choosing $a_1=a_2=1$, which clearly satisfy the respective transport equation in \eqref{eq:trans_eq_u1_u2}, we get
\begin{equation}
\label{eq:I1_Fourier}
I_1\to i\mu \cdot \int_{\Omega \cup \Omega^\ast}  \tilde A e^{-ix\cdot \xi}dx=i\cF(d\tilde A)(\xi), \quad h\to 0,
\end{equation}
where $d\tilde A$ is a 2-form defined by \eqref{eq:def_mag_field}.

Turning attention to $I_2$, we argue similarly as above to obtain
\begin{equation}
\label{eq:I2_Fourier}
\begin{aligned}
I_2&\to -i\mu\cdot  \int_{\Omega\cup \Omega^\ast}\tilde A (e^{-ix\cdot \xi_-}+  e^{-ix\cdot \xi_+})dx
\\
&=-i\cF(d\tilde A)(\xi_-)-i\cF(d\tilde A)(\xi_+), \quad h\to 0.
\end{aligned}
\end{equation}
for any $\mu \in \mathrm{Span}\{\mu^{(1)}, \mu^{(2)}\}$ and all $\xi \in \R^n$ that satisfy \eqref{eq:condition_vectors}. By Lemma \ref{lem:Riemann_Lebesgue} and the definition of $\xi_\pm$ given by \eqref{eq:def_pm_xi}, there exist constant $C>0$, $\varepsilon_0>0$, and $0<\alpha<1$ such that the estimate 
\begin{equation}
\label{eq:est_I2}
|I_2|\le C\big(e^{-\frac{\varepsilon^2}{\pi}\cdot \frac{|\xi'|^2}{h^2|\xi|^2}}+\varepsilon^\alpha\big)
\end{equation}
is valid for any $0<\varepsilon<\varepsilon_0$.

Finally, we estimate $I_3$. Note that estimate \eqref{eq:est_r_domain} yields
\begin{equation}
\label{eq:est_remainder_derivative}
\|Dr_j\|_{L^2(\Omega \cup \Omega^\ast)}=\cO(1), \quad h\to 0.
\end{equation}
Since the amplitudes $a_j\in C^\infty(\overline{\Omega})$, $j=1,2$, we utilize estimates \eqref{eq:est_remainder} and \eqref{eq:est_remainder_derivative}, in conjunction with the Cauchy-Schwartz inequality, to deduce that
\begin{equation}
\label{eq:est_I3}
|I_3|\le Ch, \quad h\to 0.
\end{equation}

We next proceed to estimate the second term on the left-hand side of the integral inequality \eqref{eq:int_id_est_hyp}. To that end, after substituting the CGO solutions $v$ and $u_2$ given by \eqref{eq:CGO_v_explicit_hyp} and \eqref{eq:CGO_u2_explicit_hyp}, respectively, into \eqref{eq:int_id_est_hyp}, we have
\begin{align*}
&h\int_\Omega q u_2\overline{v}dx
\\
= &h \int_{\Omega}q [e^{-ix\cdot \xi}(\overline{a_1}a _2+\overline{a_1}r_2+a_2\overline{r_1}+\overline{r_1}r_2)
+e^{-ix^\ast\cdot \xi}(\overline{a_1^\ast} a_2^\ast+\overline{a_1^\ast} r_2^\ast+a_2^\ast\overline{r_1^\ast}+\overline{r_1^\ast}r_2^\ast)
\\
&\quad -e^{-ix\cdot \xi_-}(\overline{a_1^\ast}a_2+\overline{a_1^\ast}r_2+a_2\overline{r_1^\ast}+\overline{r_1^\ast}r_2)
-e^{-ix \cdot \xi_+}(\overline{a_1} a_2^\ast+\overline{a_1} r_2^\ast+a_2^\ast\overline{r_1}+\overline{r_1}r_2^\ast)]dx.
\end{align*}
As $a_j\in C^\infty(\overline{\Omega})$, $j=1,2$, we apply estimate \eqref{eq:est_remainder} and the Cauchy-Schwartz inequality to conclude that the following estimate is valid
\begin{equation}
\label{eq:est_qterm}
\bigg|h\int_\Omega (q_2-q_1) u_2\overline{v}dx\bigg| \le Ch, \quad h\to 0.
\end{equation}
Therefore, by combining \eqref{eq:I1_Fourier}, \eqref{eq:est_I2}, \eqref{eq:est_I3}, and \eqref{eq:est_qterm}, we deduce from the integral inequality \eqref{eq:int_id_est_hyp} that
\begin{equation}
\label{eq:est_Fourier_dA}
|\cF(d\tilde A)(\xi)|\le C\big(e^{\frac{9R}{h}}\|\Lambda_{A^{(1)}, q_1}^\Gamma-\Lambda_{A^{(2)}, q_2}^\Gamma\|+h+e^{-\frac{\varepsilon^2}{\pi}\cdot \frac{|\xi'|^2}{h^2|\xi|^2}}+\varepsilon^\alpha\big), \quad h\to 0.
\end{equation}

We are now ready to derive a stability estimate for  $\|dA\|_{H^{-1}(\Omega)}$. With a parameter $\rho>0$ to be specified later, let us consider the set
\begin{equation}
\label{eq:def_Erho}
E(\rho)=\{\xi\in \R^n: |\xi'|\le \rho, \: |\xi_n|\le \rho\}.
\end{equation}
Then an application of the Parseval's formula gives us
\begin{equation}
\label{eq:H-1norm_split}
\|d\tilde{A}\|_{H^{-1}(\R^n)}^2\le \int_{E(\rho)}\frac{|\cF(d\tilde{A})(\xi)|^2}{1+|\xi|^2}d\xi +\int_{\R^n\setminus E(\rho)}\frac{|\cF(d\tilde{A})(\xi)|^2}{1+|\xi|^2}d\xi.
\end{equation}

Let us estimate the terms on the right-hand side of \eqref{eq:H-1norm_split}. For the second term, by the Plancherel theorem, we have
\begin{equation}
\label{eq:est_dA_large_rho}
\begin{aligned}
\int_{\R^n\setminus E(\rho)}\frac{|\cF(d\tilde{A})(\xi)|^2}{1+|\xi|^2}d\xi 
&\le C \int_{\R^n\setminus E(\rho)}\frac{|\cF(d\tilde{A})(\xi)|^2}{1+\rho^2}d\xi
\\
&\le \frac{C}{\rho^2}\|d\tilde{A}\|_{L^2(\R^n)}
\\
&\le \frac{C}{\rho^2}.
\end{aligned}
\end{equation}

We now turn to estimate the first term. Observe that the inequality $|\xi|^2\le 2\rho^2$ holds for all $\xi\in E(\rho)$. Hence, we get 
\[
e^{-\frac{\varepsilon^2}{\pi}\cdot \frac{|\xi'|^2}{h^2|\xi|^2}}\le e^{-\frac{\varepsilon^2}{\pi}\cdot \frac{|\xi'|^2}{2h^2\rho^2}}.
\]
Since the set $\{\xi\in \R^n: |\xi'|=0\}$ is of zero Lebesgue measure, we estimate the first term on the right-hand side of \eqref{eq:H-1norm_split} as follows:
\begin{equation}
\label{eq:est_dA_small_rho}
\begin{aligned}
&\int_{E(\rho)}\frac{|\cF(d\tilde{A})(\xi)|^2}{1+|\xi|^2}d\xi 
\\
&\le 
C[e^{\frac{18R}{h}}\|\Lambda_{A^{(1)}, q_1}^\Gamma-\Lambda_{A^{(2)}, q_2}^\Gamma\|^2
+
h^2
+
\varepsilon^{2\alpha}] \int_{E(\rho)}\frac{1}{1+|\xi|^2}d\xi
+ 
C\int_{E(\rho)}\frac{e^{-\frac{\varepsilon^2}{\pi}\cdot \frac{|\xi'|^2}{h^2\rho^2}}}{1+|\xi|^2}d\xi
\\
&\le C\rho^n[e^{\frac{18R}{h}}\|\Lambda_{A^{(1)}, q_1}^\Gamma-\Lambda_{A^{(2)}, q_2}^\Gamma\|^2
+
h^2+\varepsilon^{2\alpha}]
+C\int_{-\rho}^{\rho}\int_{B'(0,\rho)}\frac{e^{-\frac{\varepsilon^2}{\pi}\cdot \frac{|\xi'|^2}{h^2\rho^2}}}{1+|\xi|^2}d\xi'd\xi_n.
\end{aligned}
\end{equation}
where $B'(0,\rho)$ is a ball in $\R^{n-1}$ centered at the origin with radius $\rho$.

We next follow the arguments in \cite[Section 3]{Choudhury_Heck} closely to estimate the integral
\[
\int_{-\rho}^{\rho}\int_{B'(0,\rho)}\frac{e^{-\frac{\varepsilon^2}{\pi}\cdot \frac{|\xi'|^2}{h^2\rho^2}}}{1+|\xi|^2}d\xi'd\xi_n.
\]
To that end, by choosing $\varepsilon>0$ such that $\varepsilon=\sqrt{h}$ and using the polar coordinates, we obtain
\begin{equation}
\label{eq:est_dA_exp_term}
\begin{aligned}
\int_{-\rho}^{\rho}\int_{B'(0,\rho)}\frac{e^{-\frac{\varepsilon^2}{\pi}\cdot \frac{|\xi'|^2}{h^2\rho^2}}}{1+|\xi|^2}d\xi'd\xi_n 
&\le 
2\rho \int_{B'(0,\rho)} e^{-\frac{\varepsilon^2}{\pi}\cdot \frac{|\xi'|^2}{h^2\rho^2}}d\xi'
\\
&\le C\rho\int_0^\rho  r^{n-2}e^{-\frac{1}{\pi h\rho^2}r^2}dr
\\
&=C\rho^nh^{\frac{n-1}{2}}\int_0^{h^{-\frac{1}{2}}}  u^{n-2}e^{-\frac{1}{\pi}u^2}du
\\
&\le C\rho^nh^{\frac{n-1}{2}}\int_0^\infty u^{n-2}e^{-\frac{1}{\pi}u^2}du\le C\rho^nh^{\frac{n-1}{2}},
\end{aligned}
\end{equation}
where we have made a change of variables $r=h^{\frac{1}{2}}\rho u$.

Therefore, by combining estimates \eqref{eq:est_dA_large_rho}--\eqref{eq:est_dA_exp_term}, we deduce from \eqref{eq:H-1norm_split} that
\[
\|d\tilde{A}\|_{H^{-1}(\R^n)}^2\le C\left(\rho^ne^{\frac{18R}{h}}\|\Lambda_{A^{(1)}, q_1}^\Gamma-\Lambda_{A^{(2)}, q_2}^\Gamma\|^2+\rho^nh^2+\rho^nh^{\alpha}
+ \rho^nh^{\frac{n-1}{2}}+\frac{1}{\rho^2}\right).
\]
Furthermore, since $0<h\ll 1$, $\frac{n-1}{2}\ge 1$, and $0<\alpha<1$, we have
\[
\|d\tilde{A}\|_{H^{-1}(\R^n)}^2\le C\left(\rho^ne^{\frac{18R}{h}}\|\Lambda_{A^{(1)}, q_1}^\Gamma-\Lambda_{A^{(2)}, q_2}^\Gamma\|^2+\rho^nh^{\alpha}+\frac{1}{\rho^2}\right).
\]

Let us now choose $\rho$ such that $\rho^nh^{\alpha}=\frac{1}{\rho^2}$, i.e., $\rho=h^{-\frac{\alpha}{n+2}}$. Then we obtain
\[
\|d\tilde{A}\|_{H^{-1}(\R^n)}^2\le C\left(h^{-\frac{n\alpha}{n+2}}e^{\frac{18R}{h}}\|\Lambda_{A^{(1)}, q_1}^\Gamma-\Lambda_{A^{(2)}, q_2}^\Gamma\|^2+h^{\frac{2\alpha}{n+2}}\right).
\]
Since $\frac{n\alpha}{n+2}<2$, we get
\begin{equation}
\label{eq:est_dA_tilde}
\|d\tilde{A}\|_{H^{-1}(\R^n)}^2\le C\left(e^{\frac{20R}{h}}\|\Lambda_{A^{(1)}, q_1}^\Gamma-\Lambda_{A^{(2)}, q_2}^\Gamma\|^2+h^{\frac{2\alpha}{n+2}}\right).
\end{equation}
From here, the claimed estimate \eqref{eq:H-1_est_dA} follows immediately from \eqref{eq:est_dA_tilde} by an application of the inequality 
\[
\|dA\|_{H^{-1}(\Omega)}\le \|d\tilde{A}\|_{H^{-1}(\R^n)}.
\]
This completes the proof of Proposition \ref{prop:est_dA_H-1norm}.
\end{proof}

For later purposes we need an estimate for $\|dA\|_{L^\infty(\Omega)}$. Recall that we assumed an a priori bound for $\|A\|_{H^s(\Omega)}$, $s>\frac{n}{2}+1$, in the definition of the admissible set $\mathcal{A}(\Omega, M)$. Thus, there exists a constant $\eta>0$ such that $s=\frac{n}{2}+2\eta$. Hence, we apply the Sobolev embedding theorem and the interpolation theorem to derive the following estimate:
\begin{equation}
\label{eq:est_dA_Linfty}
\begin{aligned}
\|dA\|_{L^\infty(\Omega)}&\le C\|dA\|_{H^{\frac{n}{2}+ \eta}(\Omega)}
\\
&\le C\|dA\|_{H^{-1}(\Omega)}^{\frac{\eta}{1+s}} \|dA\|_{H^s(\Omega)}^{\frac{1+s-\eta}{1+s}}
\\
&\le C\|dA\|_{H^{-1}(\Omega)}^{\frac{\eta}{1+s}}
\\
&\le C\left(e^{\frac{10R}{h}}\|\Lambda_{A^{(1)}, q_1}^\Gamma-\Lambda_{A^{(2)}, q_2}^\Gamma\|^{\frac{\eta}{1+s}}+h^{\frac{\alpha \eta}{(n+2)(1+s)}}\right).
\end{aligned}
\end{equation}

\subsection{Estimating the magnetic potential}
In this subsection we verify estimate \eqref{eq:estimate_A_Linfty}, thus complete the proof of Theorem \ref{thm:estimate_A_Linfty}.  The key step is to apply a special Hodge decomposition of a vector field \cite[Theorem 3.3.2]{Sharafutdinov}.
\begin{lem}
\label{lem:vf_decomposition}
Let $\Omega\subseteq\R^n$, $n\ge 3$, be a bounded domain with smooth boundary $\p \Omega$, and let $A\in W^{1,\infty}(\Omega, \C^n)$ be a complex-valued vector field. Then there exist uniquely determined vector field $A^{\mathrm{sol}}\in W^{1,\infty}(\Omega, \C^n)$ and function $\varphi\in W^{2,\infty}(\Omega, \C)$ such that
\begin{equation}
\label{eq:vf_decomposition}
A=A^{\mathrm{sol}}+\nabla \varphi, \quad \div A^{\mathrm{sol}}=0, \quad \varphi|_{\p \Omega}=0.
\end{equation}
Furthermore, the following inequalities 
\begin{equation}
\label{eq:est_vf_decomposition}
\|\varphi\|_{W^{2,\infty}(\Omega)}\le C\|A\|_{W^{1,\infty}(\Omega)}, \quad \|A^{\mathrm{sol}}\|_{W^{1,\infty}(\Omega)}\le C\|A\|_{W^{1,\infty}(\Omega)},
\end{equation}
are valid, where the constant $C$ is independent of $A$.
\end{lem}

In particular, the function $\varphi$ in Lemma \ref{lem:vf_decomposition} satisfies
\[
\begin{cases}
\Delta \varphi=\div A \quad \text{in} \quad \Omega,\\
\varphi=0 \quad \text{on} \quad \p \Omega.
\end{cases}
\]
Hence, by the Morrey's inequality and \cite[Lemma 6.2]{Tzou_stability}, we get from estimate \eqref{eq:est_dA_Linfty} that
\begin{equation}
\label{eq:est_solenoidal}
\|A^{\mathrm{sol}}\|_{L^\infty(\Omega)}
\le 
C\|dA\|_{L^\infty(\Omega)}
\le 
C\left(e^{\frac{10R}{h}}\|\Lambda_{A^{(1)}, q_1}^\Gamma-\Lambda_{A^{(2)}, q_2}^\Gamma\|^{\frac{\eta}{1+s}}+h^{\frac{\alpha \eta}{(n+2)(1+s)}}\right),
\end{equation}
where $s=\frac{n}{2}+2\eta$.

Thanks to \eqref{eq:est_solenoidal}, we may estimate $\|A\|_{L^\infty(\Omega)}$ by
\[
\|A\|_{L^\infty(\Omega)}
\le 
\|A^{\mathrm{sol}}\|_{L^\infty(\Omega)}+\|\nabla \varphi\|_{L^\infty(\Omega)} 
\le 
\|dA\|_{L^\infty(\Omega)}+\|\nabla \varphi\|_{L^\infty(\Omega)}.
\]
In order to verify estimate \eqref{eq:estimate_A_Linfty}, we need to estimate $\|\nabla \varphi\|_{L^\infty(\Omega)}$. This is accomplished in the following lemma.
\begin{lem}
\label{lem:est_test_function_gradient}
Let $\varphi\in W^{2,\infty}(\Omega)$ be the function from Lemma \ref{lem:vf_decomposition}, and let $\eta>0$ be a constant such that $s=\frac{n}{2}+2\eta$. Then there exists a constant $h_0>0$ such that the estimate
\begin{equation}
\label{eq:Linfty_norm_gradient_test_function}
\|\nabla \varphi\|_{L^\infty(\Omega)} 
\le 
Ce^{\frac{11R}{h}}\|\Lambda_{A^{(1)}, q_1}^\Gamma-\Lambda_{A^{(2)}, q_2}^\Gamma\|^{\frac{\eta^2}{(1+s)^2}}+Ch^{\frac{2\alpha \eta^2}{(n+2)^2(1+s)^2}}
\end{equation}
holds for all $h<h_0$ and some $\alpha \in (0,1)$. Here $R>0$ is a constant such that $\Omega \subseteq B(0,R)$.
\end{lem}
\begin{proof}
We shall follow similar ideas as in the proof of \cite[Lemma 3.6]{Ma_Liu_stability}. We extend $\varphi$ by zero to $\R^n\setminus \Omega$ and denote the extension by the same letter. Our starting point is the analysis of the integrals appearing in \eqref{eq:int_id_substitute_A} in the limit $h\to 0$. By arguing similarly as in the proof of Proposition \ref{prop:est_dA_H-1norm}, we get
\begin{equation}
\label{eq:I1I2_est}
\bigg|(\mu^{(1)}-i\mu^{(2)})\cdot \int_{\Omega \cup \Omega^\ast}  \tilde A e^{-ix\cdot \xi}\overline{a_1}a_2dx\bigg|
\le 
C\big(e^{\frac{9R}{h}}\|\Lambda_{A^{(1)}, q_1}^\Gamma-\Lambda_{A^{(2)}, q_2}^\Gamma\|+e^{-\frac{\varepsilon^2}{\pi}\cdot \frac{|\xi'|^2}{h^2|\xi|^2}}+\varepsilon^\alpha+h\big).
\end{equation}

Let us now choose the amplitudes $a_1=1$ and $a_2$ such that
\begin{equation}
\label{eq:transport_a2}
((\mu^{(1)} - i\mu^{(2)}) \cdot \nabla)a_2(x, \mu^{(1)} - i\mu^{(2)})=1 \quad \text{ in } \quad \Omega.
\end{equation}
It is straightforward to verify that this choice of $a_1$ and $a_2$ satisfies the respective transport equation in \eqref{eq:trans_eq_u1_u2}.

In view of Lemma \ref{lem:vf_decomposition} applied to the domain $\Omega \cup \Omega^\ast$, we may write $\tilde A=\tilde A^{\mathrm{sol}}+\nabla \tilde \varphi$, where the function $\tilde \varphi \in W^{2,\infty}(\Omega \cup \Omega^\ast)$ satisfies $\tilde \varphi|_{\p (\Omega \cup \Omega^\ast)}=0$. This allows us to rewrite the left-hand side of \eqref{eq:I1I2_est} as
\begin{align*}
&(\mu^{(1)}-i\mu^{(2)})\cdot \int_{\Omega \cup \Omega^\ast}  \tilde A e^{-ix\cdot \xi}a_2dx
\\
&= (\mu^{(1)}-i\mu^{(2)})\cdot \int_{\Omega \cup \Omega^\ast}  \tilde A^{\mathrm{sol}} e^{-ix\cdot \xi}a_2dx
+
(\mu^{(1)}-i\mu^{(2)})\cdot \int_{\Omega \cup \Omega^\ast}  \nabla \tilde  \varphi e^{-ix\cdot \xi}a_2dx.
\end{align*}

Since $a_2\in C^\infty(\overline{\Omega})$, we get
\[
\bigg|(\mu^{(1)}-i\mu^{(2)})\cdot \int_{\Omega \cup \Omega^\ast}  \tilde A^{\mathrm{sol}} e^{-ix\cdot \xi}a_2dx\bigg|\le C\|\tilde{A}^{\mathrm{sol}}\|_{L^\infty(\Omega \cup \Omega^\ast)}.
\]
On the other hand, as the amplitude $a_2$ satisfies \eqref{eq:transport_a2}, the function $\tilde \varphi$ satisfies $\tilde \varphi|_{\p (\Omega \cup \Omega^\ast)}=0$, and $\mu^{(1)}\cdot \xi=\mu^{(2)}\cdot \xi=0$, we integrate by parts to obtain
\[
(\mu^{(1)}-i\mu^{(2)})\cdot \int_{\Omega \cup \Omega^\ast}  \nabla \tilde{\varphi} e^{-ix\cdot \xi}a_2dx=-\int_{\Omega \cup \Omega^\ast} \tilde \varphi e^{-ix\cdot \xi}dx=-\mathcal{F}(\tilde \varphi)(\xi).
\]
Therefore, it follows from estimate \eqref{eq:I1I2_est} that
\[
|\mathcal{F}(\tilde \varphi)(\xi)|
\le 
C\big(e^{\frac{9R}{h}}\|\Lambda_{A^{(1)}, q_1}^\Gamma-\Lambda_{A^{(2)}, q_2}^\Gamma\|+h+e^{-\frac{\varepsilon^2}{\pi}\cdot \frac{|\xi'|^2}{h^2|\xi|^2}}+\varepsilon^\alpha+\|\tilde{A}^{\mathrm{sol}}\|_{L^\infty(\Omega \cup \Omega^\ast)}\big).
\]

To estimate $\|\tilde{A}^{\mathrm{sol}}\|_{L^\infty(\Omega \cup \Omega^\ast)}$, by the Morrey's inequality, we have
\[
\|\tilde{A}^\mathrm{sol}\|_{L^\infty(\Omega \cup \Omega^\ast)}\le \|d\tilde A\|_{L^\infty(\R^n)}.
\]
On the other hand, by arguing similarly as in estimate \eqref{eq:est_dA_Linfty}, we get from estimate  \eqref{eq:est_dA_tilde} that
\begin{equation}
\label{eq:est_dA_tilde_Linfty}
\|d\tilde A\|_{L^\infty(\R^n)}
\le 
C\left(e^{\frac{10R}{h}}\|\Lambda_{A^{(1)}, q_1}^\Gamma-\Lambda_{A^{(2)}, q_2}^\Gamma\|^{\frac{\eta}{1+s}}
+
h^{\frac{\alpha \eta}{(n+2)(1+s)}}\right).
\end{equation}
Due to the inequalities $0<h\ll 1$ and $\frac{\alpha \eta}{(n+2)(1+s)}<1$, we obtain
\begin{equation}
\label{eq:est_Fourier_phi}
|\mathcal{F}(\tilde \varphi)(\xi)|
\le 
C\left(e^{\frac{10R}{h}}\|\Lambda_{A^{(1)}, q_1}^\Gamma-\Lambda_{A^{(2)}, q_2}^\Gamma\|^{\frac{\eta}{1+s}}+h^{\frac{\alpha \eta}{(n+2)(1+s)}}+e^{-\frac{\varepsilon^2}{\pi}\cdot \frac{|\xi'|^2}{h^2|\xi|^2}}+\varepsilon^\alpha\right).
\end{equation}

We next proceed similarly as in the proof of Proposition \ref{prop:est_dA_H-1norm} to estimate $\|\varphi\|_{H^{-1}(\Omega)}$. Let $\rho$ be a parameter that we shall choose later. Then we apply the Parseval's formula to write
\[
\|\tilde\varphi\|_{H^{-1}(\R^n)}^2\le \int_{E(\rho)}\frac{|\cF(\tilde\varphi)(\xi)|^2}{1+|\xi|^2}d\xi +\int_{\R^n\setminus E(\rho)}\frac{|\cF(\tilde\varphi)(\xi)|^2}{1+|\xi|^2}d\xi,
\]
where the set $E(\rho)$ is given by \eqref{eq:def_Erho}. By following the same computations as in estimate \eqref{eq:est_dA_large_rho}, we have
\begin{equation}
\label{eq:est_phi_large_rho}
\int_{\R^n\setminus E(\rho)}\frac{|\cF(\tilde\varphi)(\xi)|^2}{1+|\xi|^2}d\xi
\le
 \frac{C}{\rho^2}.
\end{equation}

By choosing $\varepsilon=\sqrt{h}$, utilizing estimate \eqref{eq:est_dA_exp_term}, in conjunction with the inequalities $0<h\ll 1$, $\frac{n-1}{2}\ge 1$, and $\frac{\eta}{(n+2)(1+s)}<1$, we deduce from \eqref{eq:est_Fourier_phi} that
\begin{equation}
\label{eq:est_phi_small_rho}
\int_{E(\rho)}\frac{|\cF(\tilde\varphi)(\xi)|^2}{1+|\xi|^2}d\xi 
\le 
C\rho^n\left(e^{\frac{20R}{h}}\|\Lambda_{A^{(1)}, q_1}^\Gamma-\Lambda_{A^{(2)}, q_2}^\Gamma\|^{\frac{2\eta}{1+s}}+h^{\frac{2\alpha \eta}{(n+2)(1+s)}}\right)
\end{equation}
for some $\alpha\in (0,1)$. 
Combining \eqref{eq:est_phi_large_rho} and \eqref{eq:est_phi_small_rho}, we conclude that
\[
\|\tilde\varphi\|_{H^{-1}(\R^n)}^2\le C\left(\rho^ne^{\frac{20R}{h}}\|\Lambda_{A^{(1)}, q_1}^\Gamma-\Lambda_{A^{(2)}, q_2}^\Gamma\|^{\frac{2\eta}{1+s}}+\rho^nh^{\frac{2\alpha \eta}{(n+2)(1+s)}}+\frac{1}{\rho^2}\right).
\]

Choosing $\rho>0$ such that $\rho^nh^{\frac{2\alpha \eta}{(n+2)(1+s)}}=\frac{1}{\rho^2}$, namely, $\rho=h^{-\frac{2\alpha \eta}{(n+2)^2(1+s)}}$, we get
\begin{align*}
\|\tilde\varphi\|_{H^{-1}(\R^n)}^2 
\le 
&C\left(h^{-\frac{2n\alpha \eta}{(n+2)^2(1+s)}}e^{\frac{20R}{h}}\|\Lambda_{A^{(1)}, q_1}^\Gamma-\Lambda_{A^{(2)}, q_2}^\Gamma\|^{\frac{2\eta}{1+s}}
+h^{\frac{4\alpha \eta}{(n+2)^2(1+s)}}\right)
\\
\le 
&C\left(e^{\frac{22R}{h}}\|\Lambda_{A^{(1)}, q_1}^\Gamma-\Lambda_{A^{(2)}, q_2}^\Gamma\|^{\frac{2\eta}{1+s}}+h^{\frac{4\alpha \eta}{(n+2)^2(1+s)}}\right).
\end{align*}
Here we have utilized the inequality $\frac{2n\alpha \eta}{(n+2)^2(1+s)}<2$.
Thus, by similar arguments as the end of the proof of Proposition \ref{prop:est_dA_H-1norm}, we have
\begin{equation}
\label{eq:est_dphi_H-1}
\|\tilde\varphi\|_{H^{-1}(\R^n)} 
\le 
C\left(e^{\frac{11R}{h}}\|\Lambda_{A^{(1)}, q_1}^\Gamma-\Lambda_{A^{(2)}, q_2}^\Gamma\|^{\frac{\eta}{1+s}}+h^{\frac{2\alpha \eta}{(n+2)^2(1+s)}}\right).
\end{equation}

Since $A=\tilde A$ in $\Omega$, we see that $\varphi|_{\Gamma}=\tilde \varphi|_{\Gamma}=0$. Furthermore, as $\varphi|_{\Gamma_0}=0$, we obtain the inequality
\[
\|\varphi\|_{H^{-1}(\Omega)}\le \|\tilde\varphi\|_{H^{-1}(\R^n)},
\]
which implies
\[
\|\varphi\|_{H^{-1}(\Omega)}
\le 
C\left(e^{\frac{11R}{h}}\|\Lambda_{A^{(1)}, q_1}^\Gamma-\Lambda_{A^{(2)}, q_2}^\Gamma\|^{\frac{\eta}{1+s}}+h^{\frac{2\alpha \eta}{(n+2)^2(1+s)}}\right).
\]

From here, we have the claimed estimate \eqref{eq:Linfty_norm_gradient_test_function} by similar computations as in estimate \eqref{eq:est_dA_Linfty}.
This completes the proof of Lemma \ref{lem:est_test_function_gradient}.
\end{proof}

We are now ready to verify estimate \eqref{eq:estimate_A_Linfty}. In view of estimates \eqref{eq:est_solenoidal} and \eqref{eq:Linfty_norm_gradient_test_function}, we have
\begin{equation}
\label{eq:est_Linfty_A}
\begin{aligned}
\|A\|_{L^\infty(\Omega)} 
\le & 
\|A^{\mathrm{sol}}\|_{L^\infty(\Omega)}+\|\nabla \varphi\|_{L^\infty(\Omega)}
\\
\le &C\left(e^{\frac{10R}{h}}\|\Lambda_{A^{(1)}, q_1}^\Gamma-\Lambda_{A^{(2)}, q_2}^\Gamma\|^{\frac{\eta}{1+s}}
+
h^{\frac{\alpha \eta}{(n+2)(1+s)}}
+e^{\frac{11R}{h}}\|\Lambda_{A^{(1)}, q_1}^\Gamma-\Lambda_{A^{(2)}, q_2}^\Gamma\|^{\frac{\eta^2}{(1+s)^2}}\right.
\\
&\left.+h^{\frac{2\alpha \eta^2}{(n+2)^2(1+s)^2}}\right)
\\
\le &C\left(e^{\frac{11R}{h}}\|\Lambda_{A^{(1)}, q_1}^\Gamma-\Lambda_{A^{(2)}, q_2}^\Gamma\|^{\frac{\eta^2}{(1+s)^2}}+h^{\frac{2\alpha \eta^2}{(n+2)^2(1+s)^2}}\right)
\\
\le &C\left(e^{\frac{11R}{h^{1/\alpha}}}\|\Lambda_{A^{(1)}, q_1}^\Gamma-\Lambda_{A^{(2)}, q_2}^\Gamma\|^{\frac{\eta^2}{(1+s)^2}}+h^{\frac{2\alpha \eta^2}{(n+2)^2(1+s)^2}}\right),
\end{aligned}
\end{equation}
where we have used the inequality $\frac{1}{h}< \frac{1}{h^{\frac{1}{\alpha}}}$ in the last step.

Let $\tilde h_0=\min \{h_0, \varepsilon_0^2\}$ and $\delta = (e^{-\frac{11R}{\tilde{h}_0^{1/\alpha}}})^{\frac{2(1+s)^2}{\eta^2}}$. Let us assume first that $\|\Lambda_{A^{(1)}, q_1}^\Gamma-\Lambda_{A^{(2)}, q_2}^\Gamma\| < \delta$. We then choose
\begin{equation}
\label{eq:power_h_alpha}
h= \bigg\{\frac{1}{11R} \big|\log \|\Lambda_{A^{(1)}, q_1}^\Gamma-\Lambda_{A^{(2)}, q_2}^\Gamma\|\big|\bigg\}^{-\frac{\alpha \eta^2}{2(1+s)^2}}.
\end{equation}
With this choice of $h$, we have
\[
\frac{1}{h^{\frac{1}{\alpha}}}=\frac{1}{11R} \big|\log \|\Lambda_{A^{(1)}, q_1}^\Gamma-\Lambda_{A^{(2)}, q_2}^\Gamma \|\big|^{\frac{\eta^2}{2(1+s)^2}}.
\]
Hence, by substituting \eqref{eq:power_h_alpha} into \eqref{eq:est_Linfty_A}, we get the estimate
\begin{equation}
\label{eq:estimate_A_hyp}
\begin{aligned}
\|A\|_{L^\infty(\Omega)}
&\le C\|\Lambda_{A^{(1)}, q_1}^\Gamma-\Lambda_{A^{(2)}, q_2}^\Gamma\|^{\frac{\eta^2}{2(1+s)^2}}+C \big|\log \|\Lambda_{A^{(1)}, q_1}^\Gamma-\Lambda_{A^{(2)}, q_2}^\Gamma\|\big|^{-\frac{\alpha^2 \eta^4}{(n+2)^2(1+s)^4}}.
\end{aligned}
\end{equation}

When $\|\Lambda_{A^{(1)}, q_1}^\Gamma-\Lambda_{A^{(2)}, q_2}^\Gamma\| \ge \delta$, we have 
\[
\|A\|_{L^\infty(\Omega)} 
\le 
2M
\le 
\frac{2CM}{\delta^{\frac{\eta^2}{2(1+s)^2}}}\delta^{\frac{\eta^2}{2(1+s)^2}}\le \frac{2CM}{\delta^{\frac{\eta^2}{2(1+s)^2}}}\|\Lambda_{A^{(1)}, q_1}^\Gamma-\Lambda_{A^{(2)}, q_2}^\Gamma\|^{\frac{\eta^2}{2(1+s)^2}}.
\]
Thus, we have verified estimate \eqref{eq:estimate_A_Linfty}. 

To complete the proof of Theorem \ref{thm:estimate_A_Linfty}, we need to verify that the choice of $h$ in \eqref{eq:power_h_alpha} yields the inequalities $h<h_0$ and $1-\frac{h^2}{4}|\xi|^2>0$ for $\xi=(\xi', \xi_n)\in \R^n$ such that $|\xi'|<\rho$ and $\xi_n<\rho$. To verify $h<h_0$, we have the following inequalities:
\begin{align*}
\|\Lambda_{A^{(1)}, q_1}^\Gamma-\Lambda_{A^{(2)}, q_2}^\Gamma\| &<\delta = (e^{-\frac{11R}{\tilde{h}_0^{1/\alpha}}})^{\frac{2(1+s)^2}{\eta^2}}\ll 1
\\
\implies  \|\Lambda_{A^{(1)}, q_1}^\Gamma-\Lambda_{A^{(2)}, q_2}^\Gamma\|^{\frac{\eta^2}{2(1+s)^2}} &< e^{-\frac{11R}{\tilde{h}_0^{1/\alpha}}}
\\
\implies \log  \|\Lambda_{A^{(1)}, q_1}^\Gamma-\Lambda_{A^{(2)}, q_2}^\Gamma\|^{\frac{\eta^2}{2(1+s)^2}} &<-\frac{11R}{\tilde{h}_0^{1/\alpha}}
\\
\implies \big| \log \|\Lambda_{A^{(1)}, q_1}^\Gamma-\Lambda_{A^{(2)}, q_2}^\Gamma\| \big|^{\frac{\eta^2}{2(1+s)^2}} &>\frac{11R}{\tilde{h}_0^{1/\alpha}}
\\
\implies \frac{1 }{11R} \big| \log \|\Lambda_{A^{(1)}, q_1}^\Gamma-\Lambda_{A^{(2)}, q_2}^\Gamma\| \big|^{\frac{\eta^2}{2(1+s)^2}} &>\frac{1}{\tilde{h}_0^{1/\alpha}}
\\
\implies \frac{1}{h^{\frac{1}{\alpha}}}&>\frac{1}{\tilde{h}_0^{1/\alpha}}
\\
\implies h&<\tilde h_0<h_0.
\end{align*}

We now turn to check the inequality $1-\frac{h^2}{4}|\xi|^2>0$. To that end, we have
\[
h^2\frac{|\xi|^2}{4}<h^2\frac{\rho^2}{2}<\frac{1}{2}h^2h^{-\frac{2\alpha \eta}{(n+2)^2(1+s)}}=\frac{1}{2}h^{2-\frac{2\alpha \eta}{(n+2)^2(1+s)}}<1,
\]
where we have utilized the inequalities $|\xi|^2<2\rho^2$ and $\frac{2\alpha \eta}{(n+2)^2(1+s)}<2$. This completes the proof of Theorem \ref{thm:estimate_A_Linfty}.

\section{Estimates for the Electric Potential}
\label{sec:proof_elec_term}

This section is devoted to proving the estimates for the electric potential, namely, \eqref{eq:estimate_q_hyp} and \eqref{eq:estimate_q_Linfty}. For convenience, in what follows we shall again denote $A=A^{(2)}-A^{(1)}$ and $q=q_2-q_1$. We begin this section by recalling the integral identity \eqref{eq:int_id_A_hyp}:
\[
\int_\Omega (A\cdot Du_2+qu_2)\overline{v}dx
=
\int_{\Gamma}\p_\nu (\Delta (u_1-u_2))\overline{v}dS
+\int_{\Gamma}\p_\nu (u_1-u_2)(\overline{\Delta v})dS.
\]
Here $v$ is a solution to the equation $\mathcal{L}^\ast_{A^{(1)},q_1}v=0$ in $\Omega$, and $u_j$, $j=1,2$, satisfies the equation $\mathcal{L}_{A^{(j)},q_j}u_j=0$ in $\Omega$.

Let us now estimate the right-hand side of the integral identity above. To that end, we apply Lemma \ref{lem:est_int_id} to obtain
\[
\bigg|\int_{\Gamma}\p_\nu (\Delta (u_1-u_2))\overline{v}dS
+\int_{\Gamma}\p_\nu (u_1-u_2)(\overline{\Delta v})dS\bigg| \le Ce^{\frac{9R}{h}}\|\Lambda_{A^{(1)}, q_1}^\Gamma-\Lambda_{A^{(2)}, q_2}^\Gamma\|.
\]

To estimate the left-hand side of the integral identity, by the Cauchy-Schwartz inequality, estimate \eqref{eq:est_solutions_hyp}, and the inequality $\frac{1}{h}\le e^{\frac{R}{h}}$, we have
\begin{align*}
\bigg|\int_\Omega A\cdot Du_2\overline{v}dx\bigg|
&\le 
\|A\|_{L^\infty(\Omega)} \|Du_2\|_{L^2(\Omega)} \|v\|_{L^2(\Omega)}
\\
&\le C \|A\|_{L^\infty(\Omega)} \cdot \frac{1}{h^4}e^{\frac{2R}{h}} \cdot \frac{1}{h}e^{\frac{2R}{h}}
\\
&\le Ce^{\frac{9R}{h}} \|A\|_{L^\infty(\Omega)}.
\end{align*}
Hence, by combining the previous two estimates, we obtain
\begin{equation}
\label{eq:est_q}
\bigg|\int_\Omega qu_2\overline{v}dx\bigg|
\le 
Ce^{\frac{9R}{h}} \big(\|A\|_{L^\infty(\Omega)}+\|\Lambda_{A^{(1)}, q_1}^\Gamma-\Lambda_{A^{(2)}, q_2}^\Gamma\|\big).
\end{equation}

We next substitute the CGO solutions \eqref{eq:CGO_v_explicit_hyp} and \eqref{eq:CGO_u2_explicit_hyp} into the left-hand side of \eqref{eq:est_q} and pass to the limit $h\to 0$. To that end, by a direct computation, we get
\begin{equation}
\label{eq:q_cgo_sub}
\begin{aligned}
&\int_\Omega q u_2\overline{v}dx
\\
&= \int_\Omega q (e^{-ix\cdot \xi}\overline{a_1}a_2+e^{-ix^\ast \cdot \xi}\overline{a_1^\ast} a_2^\ast) dx
+
\int_\Omega q(-e^{-ix\cdot \xi_-}\overline{a_1^\ast}a_2-e^{-ix \cdot \xi_+}\overline{a_1} a_2^\ast) dx
+\int_\Omega qw_2dx
\\
&=I_1+I_2+I_3,
\end{aligned}
\end{equation}
where $\xi_\pm$ is given by \eqref{eq:def_pm_xi}, and
\begin{align*}
w_2 =&e^{-ix\cdot \xi}(\overline{a_1}r_2+a_2\overline{r_1}+\overline{r_1}r_2)
+e^{-ix^\ast \cdot \xi}(\overline{a_1^\ast} r_2^\ast+a_2^\ast\overline{r_1^\ast}+\overline{r_1^\ast}r_2^\ast)
\\
&-e^{-ix\cdot \xi_-}(\overline{a_1^\ast}r_2+a_2\overline{r_1^\ast}+\overline{r_1^\ast}r_2)
-e^{-ix \cdot \xi_+}(\overline{a_1} r_2^\ast+a_2^\ast\overline{r_1}+\overline{r_1}r_2^\ast).
\end{align*}

Let us now analyze each integral in \eqref{eq:q_cgo_sub} in the limit $h\to 0$. We extend $q$ by zero on $\R^n \setminus \Omega$ and denote the extension by the same letter. For $I_3$, since $q\in \mathcal{Q}(\Omega, M)$, and the amplitudes $a_j\in C^\infty(\overline{\Omega})$, $j=1,2$, we utilize estimate \eqref{eq:est_r_domain} to deduce
\begin{equation}
\label{eq:est_rest}
|I_3|\le Ch, \quad h\to 0.
\end{equation}

Turning attention to $I_1$, by choosing $a_1=a_2=1$ and making a change of variables, we have
\begin{equation}
\label{eq:Fou_q}
I_1
= \int_{\Omega \cup \Omega^\ast}
e^{-ix\cdot \xi}\tilde qdx=\mathcal{F}(\tilde q)(\xi),
\end{equation}
where the function $\tilde q$ is the even extension of $q$ given by \eqref{eq:extension_q}.

We finally estimate $I_2$. To that end, we set $a_1=a_2=1$ again to obtain
\[
I_2=\int_{\Omega \cup \Omega^\ast} q(-e^{-ix\cdot \xi_+}-e^{-ix \cdot \xi_-}) dx=-\mathcal{F}(\tilde{q})(\xi_+)-\mathcal{F}(\tilde{q})(\xi_-).
\] 
Hence, it follows from Lemma \ref{lem:Riemann_Lebesgue} and the definition of $\xi_\pm$ given by \eqref{eq:def_pm_xi} that  there exist constants $C>0$, $\varepsilon_0>0$, and $0<\alpha<1$ such that the estimate 
\begin{equation}
\label{eq:est_I2_q}
|I_2|\le C\big(e^{-\frac{\varepsilon^2}{\pi}\cdot \frac{|\xi'|^2}{h^2|\xi|^2}}+\varepsilon^\alpha\big)
\end{equation}
is valid for any $0<\varepsilon<\varepsilon_0$.

Hence, it follows from \eqref{eq:est_q}--\eqref{eq:est_I2_q} that 
\begin{equation}
\label{eq:est_Fourier_q}
|\mathcal{F}(\tilde q)(\xi)|
\le 
C\left(e^{\frac{9R}{h}} \|A\|_{L^\infty(\Omega)}
+
e^{\frac{9R}{h}} \|\Lambda_{A^{(1)}, q_1}^\Gamma-\Lambda_{A^{(2)}, q_2}^\Gamma\|
+
e^{-\frac{\varepsilon^2}{\pi}\cdot \frac{|\xi'|^2}{h^2|\xi|^2}}
+
\varepsilon^\alpha
+
h\right).
\end{equation}
Furthermore, by substituting estimate \eqref{eq:est_Linfty_A} into \eqref{eq:est_Fourier_q} and applying the equality $s=\frac{n}{2}+2\eta$, we get
\begin{equation}
\label{eq:est_Fourier_q_sub}
\begin{aligned}
|\mathcal{F}(\tilde q)(\xi)|
\le 
&C\left(e^{\frac{20R}{h}} \|\Lambda_{A^{(1)}, q_1}^\Gamma-\Lambda_{A^{(2)}, q_2}^\Gamma\|^{\frac{\eta^2}{(1+s)^2}}
+h^{\frac{2\alpha \eta^2}{(n+2)^2(1+s)^2}}
+
e^{-\frac{\varepsilon^2}{\pi}\cdot \frac{|\xi'|^2}{h^2|\xi|^2}}
+
\varepsilon^\alpha
+
h\right).
\end{aligned}
\end{equation}

We are now ready to establish a stability estimate for $\|\tilde{q}\|_{H^{-1}(\R^n)}$. Let $\rho>0$ be a parameter to be chosen later, and consider again the set $E(\rho)$ given by \eqref{eq:def_Erho}. By the Parseval's formula, we write
\begin{equation}
\label{eq:H-1norm_split_q}
\|\tilde{q}\|_{H^{-1}(\R^n)}^2
\le 
\int_{E(\rho)}\frac{|\cF(\tilde{q})(\xi)|^2}{1+|\xi|^2}d\xi 
+
\int_{\R^n\setminus E(\rho)}\frac{|\cF(\tilde{q})(\xi)|^2}{1+|\xi|^2}d\xi.
\end{equation}
Arguing similarly as in \eqref{eq:est_dA_large_rho}, we get
\begin{equation}
\label{eq:est_q_large_rho}
\int_{\R^n\setminus E(\rho)}\frac{|\cF(\tilde{q})(\xi)|^2}{1+|\xi|^2}d\xi 
\le 
\frac{C}{\rho^2}.
\end{equation}

We next follow similar arguments as in Section 3 to estimate first term on the right-hand side of \eqref{eq:H-1norm_split_q}, and provide the proof for the sake of completeness.  To that end, we have
\begin{equation}
\label{eq:est_q_small_rho_1}
\int_{E(\rho)}\frac{|\cF(\tilde{q})(\xi)|^2}{1+|\xi|^2}d\xi 
\le 
C\rho^n \left(e^{\frac{40R}{h}} \|\Lambda_{A^{(1)}, q_1}^\Gamma-\Lambda_{A^{(2)}, q_2}^\Gamma\|^{\frac{2\eta^2}{(1+s)^2}}
+h^{\frac{4\alpha \eta^2}{(n+2)^2(1+s)^2}}
+
h^{\frac{n-1}{2}}
+
h^{\alpha}
+
h^2\right),
\end{equation}
where we have applied estimate \eqref{eq:est_dA_exp_term} and chosen $\varepsilon>0$ such that $\varepsilon=\sqrt{h}$.

Due to the inequalities $0<h\ll 1$, $\frac{n-1}{2}\ge 1$, $0<\alpha<1$, and $\frac{4 \eta^2}{(n+2)^2(1+s)^2}<1$, we get
\begin{equation}
\label{eq:est_q_small_rho}
\begin{aligned}
&\int_{E(\rho)}\frac{|\cF(\tilde{q})(\xi)|^2}{1+|\xi|^2}d\xi 
&\le  C\rho^n\left(e^{\frac{40R}{h}} \|\Lambda_{A^{(1)}, q_1}^\Gamma-\Lambda_{A^{(2)}, q_2}^\Gamma\|^{\frac{2\eta^2}{(1+s)^2}}
+ 
h^{\frac{4 \alpha \eta^2}{(n+2)^2(1+s)^2}}\right).
\end{aligned}
\end{equation}
Thus, it follows from \eqref{eq:est_q_large_rho} and \eqref{eq:est_q_small_rho} that
\begin{equation}
\label{eq:H-1_norm_q}
\begin{aligned}
\|\tilde{q}\|_{H^{-1}(\R^n)}^2
\le
&C\left(\rho^ne^{\frac{40R}{h}} \|\Lambda_{A^{(1)}, q_1}^\Gamma-\Lambda_{A^{(2)}, q_2}^\Gamma\|^{\frac{2\eta^2}{(1+s)^2}}
+h^{\frac{4\alpha  \eta^2}{(n+2)^2(1+s)^2}}
+\frac{1}{\rho^2}\right). 
\end{aligned}
\end{equation}

By choosing $\rho=h^{-\frac{4 \alpha  \eta^2}{(n+2)^3(1+s)^2}}$, which equates $\rho^nh^{\frac{4 \alpha \eta^2}{(n+2)^2(1+s)^2}}$ and $\frac{1}{\rho^2}$, we get
\begin{align*}
\|\tilde{q}\|_{H^{-1}(\R^n)}^2
\le 
&C\left(h^{\frac{-4n\alpha  \eta^2}{(n+2)^2(1+s)^2}}e^{\frac{40R}{h}} \|\Lambda_{A^{(1)}, q_1}^\Gamma-\Lambda_{A^{(2)}, q_2}^\Gamma\|^{\frac{2\eta^2}{(1+s)^2}}
+h^{\frac{8 \alpha \eta^2}{(n+2)^3(1+s)^2}}\right).
\end{align*}
Since $\frac{4n\alpha  \eta^2}{(n+2)^2(1+s)^2}<4$, we have
\begin{equation}
\label{eq:est_q_tilde}
\begin{aligned}
\|\tilde{q}\|_{H^{-1}(\R^n)}
&\le 
C\left(e^{\frac{22R}{h}}\|\Lambda_{A^{(1)}, q_1}^\Gamma-\Lambda_{A^{(2)}, q_2}^\Gamma\|^{\frac{\eta^2}{(1+s)^2}}
+
h^{\frac{4 \alpha \eta^2}{(n+2)^3(1+s)^2}}\right)
\\
&\le 
C\left(e^{\frac{22R}{h^{1/\alpha}}}\|\Lambda_{A^{(1)}, q_1}^\Gamma-\Lambda_{A^{(2)}, q_2}^\Gamma\|^{\frac{\eta^2}{(1+s)^2}}
+
h^{\frac{4 \alpha \eta^2}{(n+2)^3(1+s)^2}}\right),
\end{aligned}
\end{equation}
where we have utilized the inequality $\frac{1}{h}< \frac{1}{h^{1/\alpha}}$ in the last step.

Let $\tilde h_0=\min\{h_0,\varepsilon_0^2\}$ and $\delta = (e^{-\frac{22R}{\tilde h _0^{1/\alpha}}})^{\frac{2(1+s)^2}{\eta^2}}$. We assume first that $\|\Lambda_{A^{(1)}, q_1}^\Gamma-\Lambda_{A^{(2)}, q_2}^\Gamma\| < \delta$. We then choose
\[
h= \bigg\{\frac{1}{22R} \big|\log \|\Lambda_{A^{(1)}, q_1}^\Gamma-\Lambda_{A^{(2)}, q_2}^\Gamma\|\big|\bigg\}^{-\frac{\alpha \eta^2}{2(1+s)^2}}.
\]
Arguing similarly as in the end of Section 3, we conclude that this choice of $h$ satisfies the inequalities $h<h_0$ and $1-\frac{h^2}{4}|\xi|^2>0$ for $\xi=(\xi', \xi_n)\in \R^n$ such that $|\xi'|<\rho$ and $|\xi_n|<\rho$. 

We then substitute this choice of $h$ into \eqref{eq:est_q_tilde} to obtain
\[
\|\tilde{q}\|_{H^{-1}(\R^n)}\le  
C\left(\|\Lambda_{A^{(1)}, q_1}^\Gamma-\Lambda_{A^{(2)}, q_2}^\Gamma\|^{\frac{\eta^2}{(1+s)^2}}
+
\log \big|\|\Lambda_{A^{(1)}, q_1}^\Gamma-\Lambda_{A^{(2)}, q_2}^\Gamma\|\big|^{-\frac{2\alpha^2\eta^4}{(n+2)^3(1+s)^4}}\right).
\]
From here, the claimed estimate \eqref{eq:estimate_q_hyp} follows from the estimate above and the inequality
\[
\|q\|_{H^{-1}(\Omega)}\le \|\tilde{q}\|_{H^{-1}(\R^n)}.
\]

When $\|\Lambda_{A^{(1)}, q_1}^\Gamma-\Lambda_{A^{(2)}, q_2}^\Gamma\| \ge \delta$, we utilize the continuous inclusions
\[
L^\infty (\Omega) \hookrightarrow L^2(\Omega) \hookrightarrow H^{-1}(\Omega)
\]
to obtain
\[
\|q\|_{H^{-1}(\Omega)}
\le 
\|q\|_{L^\infty(\Omega)} 
\le 
2M
\le 
\frac{2CM}{\delta^{\frac{\eta^2}{2(1+s)^2}}}\delta^{\frac{\eta^2}{2(1+s)^2}}\le \frac{2CM}{\delta^{\frac{\eta^2}{2(1+s)^2}}}\|\Lambda_{A^{(1)}, q_1}^\Gamma-\Lambda_{A^{(2)}, q_2}^\Gamma\|^{\frac{\eta^2}{2(1+s)^2}}.
\]
This completes the proof of Theorem \ref{thm:estimate_q_hyp}.

Finally,  we observe that estimate \eqref{eq:estimate_q_Linfty} follows from  estimate \eqref{eq:estimate_q_hyp} via similar computations as in estimate \eqref{eq:est_dA_Linfty}. This completes the proof of Corollary \ref{cor:estimate_q_Linfty}.

\bibliographystyle{abbrv}
\bibliography{bibliography_biharmonic}

\begin{thebibliography}{10}

\bibitem{Agranovich}
M.~Agranovich.
\newblock {\em Sobolev Spaces, Their Generalizations and Elliptic Problems in Smooth and Lipschitz Domains}.
\newblock Springer Monographs in Mathematics. Springer Cham, 2015.

\bibitem{Alessandrini}
G.~Alessandrini.
\newblock Stable determination of conductivity by boundary measurements.
\newblock {\em Applicable Analysis}, 27:153--172, 1988.

\bibitem{Aroua_Bellassoued}
N.~Aroua and M.~Bellassoued.
\newblock Stable determination of a second order perturbation of the polyharmonic operator by boundary measurements.
\newblock {\em Journal of Mathematical Analysis and Applications}, 522(2):126965, 2023.

\bibitem{Ashbaugh}
M.~Ashbaugh.
\newblock {\em On universal inequalities for the low eigenvalues of the buckling problem.}, volume 362 of {\em Partial differential equations and inverse problems, Contemporary Mathematics}, pages 13--31.
\newblock American Mathematical Society, 2004.

\bibitem{Assylbekov_16}
Y.~Assylbekov.
\newblock Inverse problems for the perturbed polyharmonic operator with coefficients in {S}obolev spaces with non-positive order.
\newblock {\em Inverse Problems}, 32(10):105009, 2016.

\bibitem{Assylbekov_Iyer}
Y.~Assylbekov and K.~Iyer.
\newblock Determining rough first order perturbations of the polyharmonic operator.
\newblock {\em Inverse Problems and Imaging}, 13(5):1045--1066, 2019.

\bibitem{Assylbekov_Yang}
Y.~Assylbekov and Y.~Yang.
\newblock Determining the first order perturbation of a polyharmonic operator on admissible manifolds.
\newblock {\em Journal of Differential Equations}, 262(1):590--614, 2017.

\bibitem{Bhattacharyya_Ghosh_19}
S.~Bhattacharyya and T.~Ghosh.
\newblock Inverse boundary value problem of determining up to a second order tensor appear in the lower order perturbation of a polyharmonic operator.
\newblock {\em Journal of Fourier Analysis and Applications}, 25:661--683, 2019.

\bibitem{Bhattacharyya_Ghosh_second_order}
S.~Bhattacharyya and T.~Ghosh.
\newblock An inverse problem on determining second order symmetric tensor for perturbed biharmonic operator.
\newblock {\em Mathematische Annalen}, 384:1--33, 2022.

\bibitem{Bhattacharyya_Krishnan_Sahoo_21}
S.~Bhattacharyya, V.~Krishnan, and S.~Sahoo.
\newblock Unique determination of anisotropic perturbations of a polyharmonic operator from partial boundary data.
\newblock {\em arXiv preprint arXiv:2111.07610}, 2021.

\bibitem{Brown_Gauthier}
R.~Brown and L.~Gauthier.
\newblock Inverse boundary value problems for polyharmonic operators with non-smooth coefficients.
\newblock {\em Inverse Problems and Imaging}, 16(4):943--966, 2022.

\bibitem{Campos}
L.~M. B.~C. Campos.
\newblock {\em Generalized Calculus with Applications to Matter and Forces}.
\newblock CRC Press, Boca Raton, 2014.

\bibitem{Caro_11}
P.~Caro.
\newblock On an inverse problem in electromagnetism with local data: stability and uniqueness.
\newblock {\em Inverse Problems and Imaging}, 5(2):297--322, 2011.

\bibitem{Choudhury_Heck}
A.~Choudhury and H.~Heck.
\newblock Stability of the inverse boundary value problem for the biharmonic operator: {L}ogarithmic estimates.
\newblock {\em Journal of Inverse and Ill-posed Problems}, 25(2):251--263, 2017.

\bibitem{Choudhury_Krishnan}
A.~P. Choudhury and V.~Krishnan.
\newblock Stability estimates for the inverse boundary value problem for the biharmonic operator with bounded potentials.
\newblock {\em Journal of Mathematical Analysis and Applications}, 431(1):300--316, 2015.

\bibitem{Chung_142}
F.~Chung.
\newblock A partial data result for the magnetic {S}chr\"odinger inverse problem.
\newblock {\em Analysis \& PDE}, 7:117--157, 2014.

\bibitem{Colton_Haddar_Piana}
D.~Colton, H.~Haddar, and M.~Piana.
\newblock The linear sampling method in inverse electromagnetic scattering theory.
\newblock {\em Inverse Problems}, 19:S105--S137, 2003.

\bibitem{Gazzola_Grunau_Sweers}
F.~Gazzola, H.~Grunau, and G.~Sweers.
\newblock {\em Polyharminic Boundary Value Problems}.
\newblock Springer-Verlag, Berlin, 2010.

\bibitem{Ghosh_Krishnan}
T.~Ghosh and V.~Krishnan.
\newblock Determination of lower order perturbations of the polyharmonic operator from partial boundary data.
\newblock {\em Applicable Analysis}, 95(11):2444--2463, 2016.

\bibitem{Grubb}
G.~Grubb.
\newblock {\em Distributions and Operators}, volume 252 of {\em Graduate Texts in Mathematics}.
\newblock Springer, 2009.

\bibitem{Ikehata}
M.~Ikehata.
\newblock A special green's function for the biharmonic operator and its application to an inverse boundary value problem.
\newblock {\em Computers \& Mathematics with Applications}, 22(4-5):53--66, 1991.

\bibitem{Isakov_91}
V.~Isakov.
\newblock Completeness of products of solutions and some inverse problems for pde.
\newblock {\em Journal of Differential Equations}, 92(2):305--316, 1991.

\bibitem{Isakov_07}
V.~Isakov.
\newblock On uniqueness in the inverse conductivity problem with local data.
\newblock {\em Inverse Problems and Imaging}, 1(1):95--105, 2007.

\bibitem{Krupchyk_Lassas_Uhlmann_bi_partial}
K.~Krupchyk, M.~Lassas, and G.~Uhlmann.
\newblock Determining a first order perturbation of the biharmonic operator by partial boundary measurements.
\newblock {\em Journal of Functional Analysis}, 262(4):1781--1801, 2012.

\bibitem{Krup_Lass_Uhl_magSchr_euc}
K.~Krupchyk, M.~Lassas, and G.~Uhlmann.
\newblock Inverse problems with partial data for a magnetic {S}chr{\"o}dinger operator in an infinite slab and on a bounded domain.
\newblock {\em Communications in Mathematical Physics}, 312:87--126, 2012.

\bibitem{Krupchyk_Lassas_Uhlmann_poly}
K.~Krupchyk, M.~Lassas, and G.~Uhlmann.
\newblock Inverse boundary value problems for the perturbed polyharmonic operator.
\newblock {\em Transactions of the American Mathematical Society}, 366(1):95--112, 2014.

\bibitem{Krupchyk_Uhlmann_poly_16}
K.~Krupchyk and G.~Uhlmann.
\newblock Inverse boundary problems for polyharmonic operators with unbounded potentials.
\newblock {\em Journal of Spectral Theory}, 6(1):145--183, 2016.

\bibitem{Li_Yao_Zhao}
P.~Li, X.~Yao, and Y.~Zhao.
\newblock Stability for an inverse source problem of the biharmonic operator.
\newblock {\em SIAM Journal on Applied Mathematics}, 81(6):2503--2525, 2021.

\bibitem{Liu_2018}
B.~Liu.
\newblock Global identifiability of low regulity fluid parameters in acoustic tomography of moving fluid.
\newblock {\em SIAM Journal on Mathematical Analysis}, 50(6):6348--6364, 2018.

\bibitem{Liu_stability}
B.~Liu.
\newblock Stability estimates in a partial data inverse boundary value problem for biharmonic operators at high frequencies.
\newblock {\em Inverse Problems and Imaging}, 14(5):783--796, 2020.

\bibitem{Ma_Liu_stability}
Y.~Ma and G.~Liu.
\newblock Stability estimates for the inverse boundary value problem for the first order perturbation of the biharmonic operator.
\newblock {\em Journal of Mathematical Analysis and Applications}, 523:127025, 2023.

\bibitem{Meleshko}
V.~V. Meleshko.
\newblock Selected topics in the history of the two-dimensional biharmonic problem.
\newblock {\em Applied Mechanics Reviews}, 56(1):33--85, 2003.

\bibitem{Nachman_88}
A.~Nachman.
\newblock Global uniqueness theorem for a two-dimensional inverse boundary value problem.
\newblock {\em Annals of Mathematics}, 143(1):71--96, 1996.

\bibitem{NakSunUlm_1995}
G.~Nakamura, Z.~Sun, and G.~Uhlmann.
\newblock Global identifiability for an inverse problem for the {S}chr\"odinger equation in a magnetic field.
\newblock {\em Mathematische Annalen}, 303(1):377--388, 1995.

\bibitem{Nakamura_Uhlmann}
G.~Nakamura and G.~Uhlmann.
\newblock Erratum: ``{G}lobal uniqueness for an inverse boundary value problem arising in elasticity".
\newblock {\em Inventiones {M}athematicae}, 152:205--207, 2003.

\bibitem{Potenciano_Machado_17}
L.~Potenciano-Machado.
\newblock Optimal stability estimates for a magnetic {S}chr{\"o}dinger operator with local data.
\newblock {\em Inverse Problems}, 33:095001, 2017.

\bibitem{Sahoo_Salo}
S.~Sahoo and M.~Salo.
\newblock The linearized {C}alder{\'o}n problem for polyharmonic operators.
\newblock {\em Journal of Differential Equations}, 360:407--451, 2023.

\bibitem{Sharafutdinov}
V.~Sharafutdinov.
\newblock {\em Integral Geometry of Tensor Fields}, volume~1 of {\em Inverse and Ill-Posed Problems Series}.
\newblock De Gruyter, Berlin, New York, 1994.

\bibitem{Sun_95}
Z.~Sun.
\newblock An inverse boundary value problem for {S}chr\"odinger operators with vector potentials.
\newblock {\em Transactions of the American Mathematical Society}, 338(2):953--969, 1993.

\bibitem{Tzou_stability}
L.~Tzou.
\newblock Stability estimates for coefficients of magnetic {S}chr{\"o}dinger equation from full and partial boundary measurements.
\newblock {\em Communications in Partial Differential Equations}, 33(11):1911--1952, 2008.

\bibitem{Yan}
L.~Yan.
\newblock Inverse boundary problems for biharmonic operators in transversally anisotropic geometries.
\newblock {\em SIAM Journal on Mathematical Analysis}, 53(6):6617--6653, 2021.

\bibitem{Yang}
Y.~Yang.
\newblock Determining the first order perturbation of a bi-harmonic operator on bounded and unbounded domains from partial data.
\newblock {\em Journal of Differential Equations}, 257(10):3607--3639, 2014.

\end{thebibliography}

\end{document}